\documentclass[11pt]{article}

\usepackage{amsmath,amsfonts,amssymb}
\usepackage{bbm}

\usepackage{lipsum}

\usepackage{xcolor}
\usepackage{ulem}
\newtheorem{defn}{Definition}[section]
\newtheorem{theo}[defn]{Theorem}
\newtheorem{lem}[defn]{Lemma}
\newtheorem{prop}[defn]{Proposition}

\newtheorem{rem}[defn]{Remark}
\newtheorem{exam}[defn]{Example}

\newenvironment{proof}{{\bf Proof }}{{\vskip 0.1cm \hfill$\Box$}}

\def\N {{\mathbb N}}

\def\R {{\mathbb R}}
\def\E{{\mathbb E}}
\def\P{{\mathbb P}}
\def\M{{\mathbb M}}

\begin{document}

%%%%%%%%%%%%%%%%%
%%%%%%%%%%%%%%%%%

\noindent
{{\Large\bf Conservativeness and uniqueness of invariant measures related to non-symmetric divergence type operators}
{\footnote{The research of Haesung Lee was supported by Basic Science Research Program through the National Research Foundation of Korea (NRF) funded by the Ministry of Education (2020R1A6A3A01096151). }} \\ \\
\bigskip
\noindent
{\bf Haesung Lee}  \\
\noindent
{\small{\bf Abstract.}  
We present conservativeness criteria for sub-Markovian semigroups generated by divergence type operators with specified infinitesimally invariant measures. The conservativeness criteria in this article are derived by $L^1$-uniqueness and imply that a given infinitesimally invariant measure becomes an invariant measure. We explore further conditions on the coefficients of the partial differential operators that ensure the uniqueness of the invariant measure beyond the case where the corresponding semigroups are recurrent. A main observation is that for conservativeness and uniqueness of invariant measures in this article, no growth conditions are required for the partial derivatives related to the anti-symmetric matrix of functions $C=(c_{ij})_{1 \leq i,j \leq d}$ that determine a part of the drift coefficient. As stochastic counterparts, our results can be applied to show not only the existence of a pathwise unique and strong solution up to infinity to a corresponding It\^{o}-SDE, but also the existence and uniqueness of invariant measures for the family of strong solutions.


\section{Introduction} \label{intro}
Conservativeness of sub-Markovian semigroups or diffusion processes generated by Dirichlet forms or second order partial differential operators have been investigated in various ways (see \cite{kha12, Da85, K86, Ta89, Da92, O92, Osh, FOT, OU17, Stu1, St99, Eb, TaTr, RoShTr, GT15, GT2, GT17, LT18, LT19, LT20} and references therein) and conservativeness of semigroups is known to be closely related to $L^1$-uniqueness (\cite[Corollary 2.2]{St99}), strong Markov uniqueness (\cite[Lemma 1.6(i)]{Eb}), existence of invariant measures (\cite[Lemma 3.3, Remark 3.4]{LT21in}) and non-explosion of the associated Markov processes (\cite[Theorem 3.11, Corollary 3.23]{LT20}). Recently, the existence and uniqueness of (infinitesimally) invariant measures associated to second order partial differential operators with rough coefficients have been investigated in \cite{LT21in} (see Introduction therein). The main purpose of this work is to develop criteria for conservativeness and for the uniqueness of invariant measures of sub-Markovian semigroups whose generator extends $(L, C_0^{\infty}(\R^d))$ where
$$
Lf =\frac{1}{2\rho}\text{\rm div}\left(\rho (A+C) \nabla f\right)+\langle \overline{\mathbf{B}}, \nabla f \rangle, \quad f \in C_0^{\infty}(\R^d).
$$
and where the coefficients of $L$ are supposed to satisfy condition {\bf (S1)} of Section \ref{section2}. We may hence consider non-symmetric divergence type operators. \\
In order to explain our work, let us first mention previous conservativeness results. For a symmetric Dirichlet form $(\mathcal{E}^{0}, D(\mathcal{E}^{0}))$ defined as the closure of
$$
\mathcal{E}^{0}(f,g)=\frac12 \int_{\R^d} \langle A \nabla f, \nabla g \rangle d\mu, \quad f, g \in C_0^{\infty}(\R^d)
$$
in $L^2(\R^d, \mu)$ where $A=(a_{ij})_{1 \leq i,j \leq d}$ is a symmetric matrix of functions and $\mu$ is a locally finite Borel measure on $(\R^d, \mathcal{B}(\R^d))$, there exists a unique sub-Markovian $C_0$-semigroup of contractions $(T^0_t)_{t>0}$ on $L^2(\R^d, \mu)$ associated with $(\mathcal{E}^0, D(\mathcal{E}^0))$ (see \cite[Chapter 1]{FOT} and \cite[Chapter I and II]{MR}).  
As a consequence of \cite[Theorem 2.2]{Ta89}, it is shown in \cite[Example 2]{Ta89} that $(T^0_t)_{t>0}$ is conservative (i.e. $T^0_t 1_{\R^d} :=\lim_{n \rightarrow \infty} T^0_t 1_{B_n} = 1$, $\mu$-a.e. for all $t>0$) %(see \cite[page 56]{FOT})) 
if $\mu=dx$ (Lebesgue-Borel measure) and $A$ satisfies that for any compact subset $K$ of $\R^d$ there exists a constant $\delta(K)>0$ such that
\begin{equation} \label{eq:2-1}
\langle A(x) \xi, \xi \rangle \geq \delta(K)\|\xi\|^2, \quad \text{$dx$-a.e. on $K$ \; for all $\xi \in \R^d$}
\end{equation}
and that for some constant $M>0$
\begin{equation} \label{eq:1}
\sup_{\xi \in \R^d \setminus \{0\}} \frac{\langle A(x) \xi, \xi \rangle}{\|\xi\|^2} \leq M (2+\|x\|)^2 \ln(2+\|x\|), \quad \forall x \in \R^d
\end{equation}
(see also \cite[Example 5.7.1]{FOT}). 
Here, the conservativeness defined for a sub-Markovian semigroup acting on $L^2$-space is considered in a different manner than the one in Definition \ref{conserdef}(iii), but actually, they are all equivalent.
In particular, \eqref{eq:1} is sharp in the sense that $(T^0_t)_{t>0}$ is not conservative if $\mu=dx$ and $A=(1+\|x\|^2) \ln(1+\|x\|)^{\beta} id$,\, $\beta>1$ (see \cite[Note 6.6]{Da85}). Another explicit criteria for conservativeness of semigroups associated with symmetric Dirichlet forms in a general framework are investigated in \cite[Theorem 4]{Stu1} by using the intrinsic metric induced by a symmetric Dirichlet form. For instance, by \cite[4.1 Theorem]{Sturm98} 
under the assumption that $A=id$, $\mu=\rho dx$ with
$\frac{1}{\rho} \in L^1_{loc}(\R^d)$ (or $\sqrt{\rho} \in H^{1,2}_{loc}(\R^d)$), the intrinsic metric induced by $(\mathcal{E}^0, D(\mathcal{E}^0))$ is the same as the Euclidean metric,
hence it follows from \cite[Theorem 4]{Stu1} that $(T^0_t)_{t>0}$ is conservative if
$$
\int_1^{\infty} \frac{r}{\ln \left(\mu(B_r) \right)} dr =\infty
$$
(cf. \cite[Example 5.7.2]{FOT}). 
In the above two cases, it follows  from the symmetry of $(T^0_t)_{t>0}$ with respect to $\mu$ that $(T^0_t)_{t>0}$ is conservative if and only if $\mu$ is an invariant measure for $(T^0_t)_{t>0}$ (cf. Proposition \ref{prop2.2}(iii)).  Moreover, in the above two cases, no partial derivatives of the components of $A=(a_{ij})_{1 \leq i,j\leq d}$ appear although these occur in the drift coefficient at least in the distributional sense. Beyond the case of symmetric Dirichlet forms, conservativeness criteria for semigroups associated with non-symmetric and sectorial Dirichlet forms (satisfying the strong sector condition as in \cite[Chapter I, (2.4)]{MR}) in a general framework are studied in \cite{TaTr}. There, adapting a probabilistic method from \cite{Ta89}, the main tool is the extended Lyons-Zheng decomposition for the associated non-symmetric Hunt process (cf. \cite{Tr08}). As special model case of \cite{TaTr}, let us mention the case where $d \geq 2$ and the sectorial Dirichlet form $(\mathcal{B}, D(\mathcal{B}))$ is defined as the closure of
$$
\mathcal{B}(f,g)= \frac12 \int_{\R^d} \langle (A+C) \nabla f, \nabla g \rangle dx, \quad f, g \in C_0^{\infty}(\R^d)
$$
in $L^2(\R^d, dx)$, where $A=(a_{ij})_{1 \leq i,j \leq d}$ is a symmetric matrix of locally bounded and measurable functions satisfying \eqref{eq:2-1} and $C=(c_{ij})_{1\leq i, j \leq d}$ is an anti-symmetric matrix of functions satisfying $-c_{ji}=c_{ij} \in H^{1,1}_{loc}(\R^d)$ for all $1 \leq i,j \leq d$ and  for any compact subset $K \subset \R^d$
$$
\max_{1 \leq i,j \leq d}|c_{ij}(x)| \leq \delta(K), \quad \text{$dx$-a.e. on $K$.}
$$
It is derived in \cite[Example 5 (a)]{TaTr} that the sub-Markovian $C_0$-semigroup of contractions $(S_t)_{t>0}$ on $L^2(\R^d, dx)$ associated to $(\mathcal{B}, D(\mathcal{B}))$ is conservative  if there exist $N_0 \in \N$ and a constant $M>0$ such that
\begin{equation} \label{eq:3}
\frac{\langle A(x) x, x \rangle}{\|x\|^2} +|\langle \nabla C(x),x \rangle| \leq M(\|x\|^2+1) \big(\ln(\|x\|^2+1)+1\big), \quad \text{$dx$-a.e.} \; x \in \R^d \setminus \overline{B}_{N_0},
\end{equation}
where for a matrix of functions $B=(b_{ij})_{1 \leq i,j \leq d}$ with $b_{ij} \in H^{1,1}_{loc}(\R^d)$ for all $1 \leq i,j \leq d$ and $\nabla B=( (\nabla B)_1, \ldots, (\nabla B)_d)$ is defined as
\begin{equation*} \label{nablamat}
(\nabla B)_i  := \sum_{j=1}^d \partial_j b_{ij}, \quad 1 \leq i \leq d.
\end{equation*}
In particular, if \eqref{eq:3} holds, then by \cite[Example 5 (a)]{TaTr} the dual semigroup $(S'_t)_{t>0}$ of $(S_t)_{t>0}$ with respect to $dx$ is also conservative, hence $dx$ is an invariant measure for $(S_t)_{t>0}$ (cf. Proposition \ref{prop2.2}(iii)). As in the symmetric case described above, growth conditions for the partial derivatives of $a_{ij}$, $1 \leq i, j \leq d$ and local regularity conditions on $a_{ij}$ $1 \leq i,j \leq d$ are not required in \eqref{eq:3}. Various other conservativeness criteria are developed in \cite{GT17} in the setting of non-sectorial Dirichlet forms. But there similarly to \eqref{eq:3} the partial derivatives related to $C=(c_{ij})_{1 \leq i,j \leq d}$ always occur. \\
Now let us explain our conservativeness criteria in Section \ref{invnewassum}. We assume {\bf (S1)} of Section \ref{section2}. Since $\mu$ as defined in {\bf (S1)} is an infinitesimally invariant measure for $(L, C_0^{\infty}(\R^d))$ (resp. $(L'^{, \mu}, C_0^{\infty}(\R^d))$), i.e.
$$
\int_{\R^d} Lf d\mu =0 \;\; \text{ \Big(resp. $\int_{\R^d} L'^{, \mu}f d\mu =0$\Big)}, \;\quad \forall f \in C_0^{\infty}(\R^d),
$$
it follows by \cite[Theorem 1.5]{St99} that there exists a sub-Markovian $C_0$-semigroup of contractions $(\overline{T}^{\mu}_t)_{t>0}$ (resp. $(\overline{T}'^{,\mu}_t)_{t>0}$) on $L^1(\R^d, \mu)$ whose generators $(\overline{L}^{\mu}, D(\overline{L}^{\mu}))$ (resp. $(\overline{L}'^{, \mu}, D(\overline{L}'^{, \mu}))$) extends $(L, C_0^{\infty}(\R^d))$ (resp. $(L'^{, \mu}, C_0^{\infty}(\R^d))$). By Riesz-Thorin interpolation, $(\overline{T}^{\mu}_t)_{t>0}$ and $(\overline{T}’^{, \mu}_t)_{t>0}$ restricted to $L^1(\R^d, \mu)_b$ can be uniquely extended to sub-Markovian $C_0$-semigroups of contractions on each $L^r(\R^d, \mu)$, $r\in [1,\infty)$, and to sub-Markovian semigroups of contractions on $L^{\infty}(\R^d,\mu)$. We denote these by $(T^{\mu}_t)_{t>0}$ and
$(T'^{, \mu}_t)_{t>0}$ respectively, independently of the $L^{r}(\R^d, \mu)$-space, $r \in [1, \infty]$ on which they are acting (see Proposition \ref{basicprop} and Remark \ref{defnofsemi}). As a consequence of Theorem \ref{theo3:1}(i), $(T^{\mu}_t)_{t>0}$ and $(T'^{, \mu}_t)_{t>0}$ are conservative and $\mu$ is an invariant measure for $(T^{\mu}_t)_{t>0}$ and $(T'^{, \mu}_t)_{t>0}$ if  there exist constants $c_1, c_2>0$, $\alpha \in (0,1)$, $\beta>0$, $M>0$, $\lambda>0$ and $N_0 \in \N$ such that $
\mu(B_r) \leq c_1r^{\beta}+c_2 \text{ for all $r>0$}$,
\begin{equation*} \label{globuni}
\quad \quad \langle A(x) \xi, \xi \rangle \geq \lambda \|\xi\|^2, \quad \text{ for all } \xi \in \R^d \text{ and } x \in \R^d
\end{equation*}
and
\begin{equation} \label{eq:4}
\max_{1 \leq i,j \leq d} |a_{ij}(x)| + \max_{1 \leq i,j \leq d} |c_{ij}(x)|+\|\overline{\mathbf{B}}(x)\|^2 \leq M\|x\|^{2\alpha},\;\;\text{ for a.e. $x\in \R^d\setminus \overline{B}_{N_0}$.}
\end{equation}
To our best knowledge, in contrast to all previous literature, no growth conditions on the partial derivatives of $c_{ij}$, $1 \leq i,j \leq d$, are required in \eqref{eq:4}. On the other hand, local regularity conditions on $a_{ij}, \, c_{ij} \in H^{1,p}_{loc}(\R^d)$ for some $p \in (d, \infty)$, $1 \leq i,j \leq d$, are required in Theorem \ref{theo3:1} since 
the elliptic regularity result for measures (\cite[Corollary 2.10]{BKR2}) is used for its proof. To prove Theorem \ref{theo3:1}(i), we derive $L^1$-uniqueness of $(L, C_0^{\infty}(\R^d))$ by using an iteration technique based on volume and coefficient growth. Moreover, we develop a criterion for $L^r$-uniquenes of $(L, C_0^{\infty}(\R^d))$, $r \in (1,2]$, which also implies recurrence and conservativeness criteria for $(T^{\mu}_t)_{t>0}$ and $(T'^{,\mu}_t)_{t>0}$ in case $\mu$ is finite (Theorem \ref{theo:3.3}). \\
In Section \ref{invasec}, assuming $\overline{\mathbf{B}}=0$ in addition to {\bf (S1)}, we further investigate explicit conditions on the coefficients of $L$ to obtain uniqueness of infinitesimally invariant measures for $(L, C_0^{\infty}(\R^d))$ and existence and uniqueness of invariant measures for $(T^{\mu}_t)_{t>0}$. Indeed it is shown in \cite[Theorems 3.15, 3.17]{LT21in}  that recurrence of $(T^{\mu}_t)_{t>0}$ is a sufficient condition to obtain the existence and uniqueness of (infinitesimally) invariant measures up to multiplicative constants. On the other hand, it is well-known that the transition semigroup of Brownian motion with $d \geq 3$ is transient but has an invariant measure which is the Lebesgue measure (for instance, see \cite[Example 3.23, Remark 3.25]{LT21in}).  Moreover, it is also known that $(T^{\mu}_t)_{t>0}$ is transient and has two invariant measures which are not represented by a constant multiple of each other (\cite[Example 4.1]{LT21in}). Therefore, it is natural to explore conditions on the coefficients of $L$ for which $(T^{\mu}_t)_{t>0}$ is transient but has a unique invariant measure. By using a Liouville-type theorem (\cite[Theorem 9.11 (i)]{F07}) induced by the elliptic Harnack inequality (\cite[Theorem 5]{Se64}), we show that condition {\bf (S2)} (in the first paragraph of Section \ref{invasec}) implies uniqueness of infinitesimally invariant measures for $(L, C_0^{\infty}(\R^d))$. This shows in particular that the Lebesgue measure is the unique infinitesimally invariant measure for the Laplace operator. Combining the conservativeness criteria developed in Section \ref{invnewassum} with the results of Section \ref{invasec}, we present examples of classes of $L$ satisfying {\bf (S2)} so that $\mu$ is the unique invariant measure for $(T^{\mu}_t)_{t>0}$ which is transient and conservative (see Theorem \ref{rectrancriter}(i), (ii), Proposition \ref{cor:3.8} and Remark \ref{rem4.6}(ii)). We also present in Theorem \ref{rectrancriter}(iii) an example of a class of $L$ for which $(T^{\mu}_t)_{t>0}$ has no invariant measures. \\
The results in this article can be applied to stochastic counterparts. Under assumption {\bf (S1)} in Section \ref{section2}, the conservativeness criteria of Section \ref{invnewassum} can be used as criteria for existence of a pathwise unique and strong solution up to infinity to the following time-homogeneous It\^{o}-SDE on $\R^d$
\begin{equation} \label{itosde}
X_t =x+\int_0^t \sigma(X_s)dW_s + \int_0^t \left(\beta^{\rho, A+C^T}+\overline{\mathbf{B}} \right)(X_s)ds,  \quad 0 \leq t < \infty,
\end{equation}
where $x \in \R^d$, $d \geq 2$, $\sigma=(\sigma_{ij})_{1 \leq i,j \leq d}$ is a matrix of continuous functions with $A=(a_{ij})_{1 \leq i,j \leq d}=\sigma \sigma^T$ and $\beta^{\rho, A+C^T}=\frac12 \nabla (A+C^T)+\frac{1}{2\rho}(A+C^T) \nabla \rho$.
Indeed, by \cite[Theorem 2.31]{LT20} (see also \cite[Proposition 3.10]{LT18}  and \cite[Theorem 6]{LT19de}), there exists a regularized semigroup $(P^{\mu}_t)_{t>0}$ such that for each $ f \in \cup_{r \in [1, \infty]} L^r(\R^d, \mu)$ it holds
$$
P^{\mu}_{\cdot} f \in C(\R^d \times (0, \infty)) \;\text{ and }\; P^{\mu}_t f = T^{\mu}_t f, \;\text{ $\mu$-a.e.} \;\;  \forall t>0.
$$
Then, by \cite[Theorem 3.11]{LT20}
there exists a diffusion process
$$
\M =  (\Omega, \mathcal{F}, (\mathcal{F}_t)_{t \ge 0}, (X_t)_{t \ge 0}, (\mathbb{P}_x)_{x \in \R^d\cup \{\Delta\}}   )
$$
with state space $\R^d$ and lifetime 
$$
\zeta=\inf\{t\ge 0\,:\,X_t=\Delta\}=\inf\{t\ge 0\,:\,X_t\notin \R^d\}, 
$$
such that for any $f \in \mathcal{B}_b(\R^d)$, $x \in \R^d$ and $t>0$
$$
\mathbb{E}_x[f(X_t)]= P^{\mu}_t f(x),
$$
where $\Delta$ is a point at infinity and $\E_x$ is an expectation with respect to $\P_x$. Moreover, by \cite[Corollary 3.23]{LT20}, $(T^{\mu}_t)_{t>0}$ is conservative if and only if $\M$ is non-explosive, i.e.
$$
\P_x(\zeta=\infty)=1 \quad \text{ for all $x \in \R^d$}.
$$
Therefore, it follows from \cite[Corollary 3.23, Theorem 3.52]{LT20} (see also \cite[Theorem 5.1]{LT18})
 that if $(T^{\mu}_t)_{t>0}$ is conservative, then on a probability space $(\widetilde{\Omega}, \widetilde{\mathcal{F}},  \widetilde{\P})$ carrying a $d$-dimensional standard Brownian motion $(\widetilde{W}_t)_{t \geq 0}$ there exists a pathwise unique and strong solution $(Y^x_t)_{t\geq 0}$ to \eqref{itosde} such that
\begin{equation} \label{unisemi}
P^{\mu}_t f(x)  = \mathbb{E}_x[f(X_t)]=\widetilde{\E}[f(Y^x_t)], \quad \forall x \in \R^d, t>0,  f \in \mathcal{B}_b(\R^d),
\end{equation}
where $\widetilde{\E}$ is the expectation with respect to $\widetilde{\P}$. In particular, for any locally finite Borel measure $\nu$ on $(\R^d, \mathcal{B}(\R^d))$ with  $\nu  \ll dx$, it follows from \eqref{unisemi} that  $\nu$ is an invariant measure for $(T^{\mu}_t)_{t>0}$ if and only if $\nu$ is an invariant measure for the family of strong solutions $(Y^x_t)_{t\geq0}$, $x \in \R^d$, i.e.
\begin{equation} \label{invdefn}
\int_{\R^d} \widetilde{\P}(Y_t^x \in A) \nu(dx) =\nu(A), \quad \forall A \in \mathcal{B}(\R^d),\, t>0.
\end{equation}
Thus, the criteria for existence and uniqueness of invariant measures for $(T^{\mu}_t)_{t>0}$ in this article are equivalently criteria for existence and uniqueness of invariant measures for the family of strong solutions $(Y^x_t)_{t\geq0}$, $x \in \R^d$ among the class of invariant measures which are locally finite and absolutely continuous with respect to $dx$.  
In Example \ref{ex:3.10}, we present an explicit example of a Brownian motion with singular drift that has a unique invariant measure.
We end the introduction by presenting the main results regarding stochastic  counterparts which are direct consequences of 
Theorem \ref{theo3:1}(i), Theorem \ref{theo:3.3}, Proposition \ref{cor:3.8} and \cite[Corollary 3.23, Theorem 3.52]{LT20}.

\begin{theo} \label{mainsto}
Assume that {\bf (S1)} in Section \ref{section2} with $\overline{\mathbf{B}}=0$ holds. Then the following (i)-(iii) hold.
\begin{itemize}
\item[(i)]
If there exist constants $c_1, c_2>0$, $\alpha \in [0,1)$, $\beta \in [0, \infty)$, $C_1>0$ and $N_0 \in \N$ such that 
\begin{equation*}
\mu(B_r) \leq c_1r^{\beta}+c_2 \;\; \text{ for all $r>0$}
\end{equation*}
and for a.e. $x\in \R^d\setminus \overline{B}_{N_0}$,
$$
\max_{1 \leq i,j \leq d} |a_{ij}(x)|+ \left( \max_{1 \leq i,j \leq d} |c_{ij}(x)| \right) \left(\inf_{\xi \in \R^d \setminus \{0\}} \frac{\langle A(x) \xi, \xi \rangle}{\|\xi\|^2} \right)^{-1} \leq C_1 \|x\|^{2\alpha},
$$
then for each $y \in \mathbb{R}^d$ and probability space $(\widetilde{\Omega}, \widetilde{\mathcal{F}},  \widetilde{\P})$ carrying a $d$-dimensional standard Brownian motion $(\widetilde{W}_t)_{t \geq 0}$
there exists a pathwise unique and strong solution  $(Y^y_t)_{t\geq 0}$ to
\begin{equation} \label{mainresde} 
Y^y_t =y+\int_0^t \sigma(Y^y_s)d\widetilde{W}_s + \int_0^t \beta^{\rho, A+C^T}  (Y^y_s)ds,  \quad 0 \leq t < \infty,
\end{equation}
where $x \in \R^d$, $d \geq 2$, $\sigma=(\sigma_{ij})_{1 \leq i,j \leq d}$ is a matrix of continuous functions with $A=(a_{ij})_{1 \leq i,j \leq d}=\sigma \sigma^T$ and $\beta^{\rho, A+C^T}=\frac12 \nabla (A+C^T)+\frac{1}{2\rho}(A+C^T) \nabla \rho$. Moreover, $\mu$ is an invariant measure for the family of strong solutions $(Y^y_t)_{t\geq0}$, $y \in \R^d$, i.e.
$$
\int_{\R^d} \widetilde{\P}(Y_t^y \in A) \mu(dy) =\mu(A), \quad \forall A \in \mathcal{B}(\R^d),\, t>0.
$$
\item[(ii)]
If $\mu$ is finite and there exists a constant $K>0$ such that for a.e. $x \in \R^d \setminus \overline{B}_{N_0}$
$$
\frac{\langle A(x)x, x \rangle}{\|x\|^2} + \left( \max_{1 \leq i,j \leq d} |c_{ij}(x)| \right) \left(\inf_{\xi \in \R^d \setminus \{0\}} \frac{\langle A(x) \xi, \xi \rangle}{\|\xi\|^2} \right)^{-1}  \leq K \left( \|x\| \cdot  \ln \|x\|  \right)^2, 
$$
then the same conclusion as in (i) holds.

\item[(iii)]
If  there exist constants $c_1, c_2>0$, $\alpha \in [0,1)$ and $\delta \geq 0$ such that
$$
\frac{c_1}{(1+\|x\|)^{2\alpha}}  \leq \rho(x) \leq c_2(1+\|x\|^{\delta}), \quad \forall x \in \R^d,
$$
$A=(a_{ij})_{1 \leq i,j \leq d}=\frac{1}{\rho} \widetilde{A}$ and $C=(c_{ij})_{1 \leq i,j \leq d}=\frac{1}{\rho} \widetilde{C}$, where
$\widetilde{A}=(\widetilde{a}_{ij})_{1 \leq i,j \leq d}$ is a symmetric matrix of functions with $\widetilde{a}_{ij} \in H^{1,p}_{loc}(\R^d) \cap C(\R^d)$ for all $1 \leq i,j \leq d$, for some $p \in (d, \infty)$, such that for some constants $\lambda, \Lambda>0$
$$
\lambda \|\xi\|^2 \leq \langle \widetilde{A}(x) \xi, \xi \rangle \leq \Lambda \|\xi\|^2, \quad \forall x, \xi \in \R^d,
$$
and $\widetilde{C}=(\widetilde{c}_{ij})_{1 \leq i,j \leq d}$ is an anti-symmetric matrix of functions with $\widetilde{c}_{ij} \in H^{1,p}_{loc}(\R^d) \cap C(\R^d)$ such that
$$
\max_{1 \leq i,j \leq d} |\widetilde{c}_{ij}(x)| \leq \Lambda, \quad \forall  x\in \R^d,
$$
then the same conclusion as in (i) holds and $\mu$ is a unique invariant measure for the family of strong solutions $(Y^y_t)_{t\geq0}$ to \eqref{mainresde} among the class of invariant measures which are locally finite and absolutely continuous with respect to $dx$, i.e. if $\nu$ is locally finite Borel measure on $(\R^d, \mathcal{B}(\R^d))$ with $\nu  \ll dx$ satisfying \eqref{invdefn}, then there exists a constant $c>0$ such that $\nu=c\mu$.

\end{itemize}
\end{theo}

\section{Preliminaries} \label{section2}
For basic notations and conventions which are not defined in this article, we refer to \cite[Notations and Conventions]{LT20}. 
Here we briefly introduce the notations and conventions mainly used in this paper. Let $\langle \cdot, \cdot \rangle$ and $\| \cdot \|$ denote the Euclidean inner product and the Euclidean norm in $\mathbb{R}^d$, respectively. For $x \in \R^d$ and $r>0$, let $B_r(x)=\{y \in \mathbb{R}^d: \|x-y\|<r  \}$ and its closure is denoted by $\overline{B}_r (x)$ and we write $B_r=B_r(0)$. Denote by $\mathcal{B}(\R^d)$ the Borel subsets of $\R^d$ or the set of Borel measurable functions $f: \mathbb{R}^d \rightarrow \mathbb{R}$. The set of bounded Borel measurable functions $f: \mathbb{R}^d \rightarrow \mathbb{R}$ is denoted by $\mathcal{B}_b(\R^d)$. If $\mathcal{A}$ is a subset of $\mathcal{B}(\R^d)$, we define $\mathcal{A}_0=\{f \in \mathcal{A}: \text{supp}(f):=\text{supp}(|f| dx)  \text{ is compact in } \R^d   \}$, $\mathcal{A}_b= \mathcal{A} \cap \mathcal{B}_b(\R^d)$ and $\mathcal{A}_{0,b}=\mathcal{A}_0 \cap \mathcal{A}_b$. For $s \in [1, \infty]$ and an open subset $U$ of $\R^d$ with a measure $\mu$ on $U$, denote by $L^s(\R^d, \mu)$ the usual $L^s$-space on $U$ with respect to $\mu$ equipped with the norm $\|\cdot \|_{L^s(U, \mu)}$ and we write $L^s(U)=L^s(U, dx)$, where $dx$ is the Lebesgue-Borel measure. For $s \in [1, \infty]$ and a measure $\mu$ on $\R^d$, let $L_{loc}^s(\R^d, \mu)$ be defined as
$$
L^s_{loc}(\R^d, \mu)=\{f \in \mathcal{B}(\R^d):  f|_{B} \in L^s(B, \mu) \text{ for any open ball $B$ in $\R^d$} \}
$$
and let $L^s_{loc}(\R^d)=L^s_{loc}(\R^d, dx)$. The set of continuous functions on $\R^d$ and the set of compactly supported and infinitely differentiable functions are denoted by $C(\R^d)$ and $C_0^{\infty}(\R^d)$, respectively.
Let $\nabla f=(\partial_1 f, \ldots, \partial_d f)$ where $\partial_j f$ is the $j$-th weak partial derivative of $f$ on $\R^d$ and let 
$\partial_{ij} f=\partial_i \partial_j f$ and $\Delta f = \sum_{i=1}^d \partial_{ii} f$. For a vector field $\mathbf{F}=(f_1, \ldots, f_d)$, let
$\text{div} \mathbf{F}= \sum_{i=1}^d \partial_i f_i$.
For an open subset $U$ of $\R^d$ and $s \in [1, \infty]$, $H^{1,s}(U)$ is  defined as 
$$
H^{1,s}(U)=\{ f \in L^s(U): \partial_i f \in L^s(U)  \,\text{ for all $i=1,\ldots, d$} \}
$$  
equipped with the norm $\|f\|_{H^{1,s}(U)}=(\|f\|^s_{L^s(U)}+\sum_{i=1}^d\|\partial_i f \|_{L^s(U)}^s)^{1/s}$,
if $s \in [1, \infty)$ and $\|f\|_{H^{1,\infty}(U)}=\|f\|_{L^{\infty}(U)}+\sum_{i=1}^d \|\partial_i f\|_{L^{\infty}(U)}$, if $s=\infty$.
For $s \in [1, \infty]$, $H^{1,s}_{loc}(\R^d)$ is defined as
$$
H^{1,s}_{loc}(\R^d)=\{f \in L^s_{loc}(\R^d): f|_B \in H^{1,s}(B) \text{ for any open ball $B$ in $\R^d$} \}.
$$
\\ \\
{\bf Throughout this article, we assume the following condition} \\ \\
{\bf (S1)}: {\it $d \geq 2$, for some $p \in (d, \infty)$, $\rho \in H^{1,p}_{loc}(\R^d) \cap C(\R^d)$ with $\rho(x)>0$ for all $x \in \R^d$  and $\mu=\rho dx$. $A=(a_{ij})_{1 \leq i,j \leq d}$ is a symmetric matrix of functions satisfying \eqref{eq:2-1} and $a_{ji}=a_{ij} \in H^{1,p}_{loc}(\R^d) \cap C(\R^d)$ for all $1 \leq i,j \leq d$. $C=(c_{ij})_{1 \leq i,j \leq d}$ is an anti-symmetric matrix of functions satisfying $-c_{ji}=c_{ij} \in H^{1,p}_{loc}(\R^d) \cap C(\R^d)$ for all $1 \leq i,j \leq d$. $\overline{\mathbf{B}} \in L^p_{loc}(\R^d, \R^d)$ satisfies $\int_{\R^d} \langle \overline{\mathbf{B}}, \nabla f \rangle d\mu=0$ for all  $f \in C_0^{\infty}(\R^d)$. 
}
\\ \\
The partial differential operators $(L, C_0^{\infty}(\R^d))$ and $(L'^{, \mu}, C_0^{\infty}(\R^d))$
are defined as
\begin{eqnarray*}
Lf &:=& \frac{1}{2\rho}\text{\rm div}\left(\rho (A+C) \nabla f\right)+\langle \overline{\mathbf{B}}, \nabla f \rangle=
\frac12 \text{\rm trace}(A \nabla^2 f) +\langle \beta^{\rho, A+C^T}+\overline{\mathbf{B}}, \nabla f  \rangle, \quad f \in C_0^{\infty}(\R^d), \\
L'^{, \mu}f &:=& \frac{1}{2\rho}\text{\rm div}\left(\rho (A+C^T) \nabla f\right)-\langle \overline{\mathbf{B}}, \nabla f \rangle=\frac12 \text{\rm trace}(A \nabla^2 f) +\langle \beta^{\rho, A+C}-\overline{\mathbf{B}}, \nabla f  \rangle, \quad f \in C_0^{\infty}(\R^d),
\end{eqnarray*}
where $\nabla^2 f(x)$ denotes the Hessian matrix of $f$ at $x$ and given a matrix of functions $B=(b_{ij})_{1 \leq i,j \leq d}$ with $b_{ij} \in H^{1,1}_{loc}(\R^d)$ for all $1 \leq i,j \leq d$, 
\begin{equation*} \label{logderi}
\beta^{\rho, B} := \frac12 \nabla B + \frac{1}{2\rho} B \nabla \rho
\end{equation*}
and
$\nabla B=( (\nabla B)_1, \ldots, (\nabla B)_d)$ is defined as
\begin{equation*}
(\nabla B)_i  := \sum_{j=1}^d \partial_j b_{ij}, \quad 1 \leq i \leq d.
\end{equation*}

\begin{prop} \label{basicprop}
Assume that {\bf (S1)} holds. Then the following (i)--(vii) hold.
\begin{itemize}
\item[(i)]
There exists a sub-Markovian $C_0$-semigroup of contractions $(\overline{T}^{\mu}_t)_{t>0}$ on $L^1(\R^d, \mu)$ whose generator
$(\overline{L}^{\mu}, D(\overline{L}^{\mu}))$ extends $(L, C_0^{\infty}(\R^d))$.

\item[(ii)]
There exists a sub-Markovian $C_0$-semigroup of contractions $(\overline{T}'^{,\mu}_t)_{t>0}$ on $L^1(\R^d, \mu)$ whose generator
$(\overline{L}'^{, \mu}, D(\overline{L}'^{, \mu}))$ extends $(L'^{, \mu}, C_0^{\infty}(\R^d))$.

\item[(iii)]
For any $f, g \in L^1(\R^d, \mu)_b$, it holds
\begin{equation} \label{dualproper}
\int_{\R^d} T^{\mu}_t f \cdot g d\mu=\int_{\R^d} f \cdot \overline{T}'^{, \mu}_t g d\mu.
\end{equation}
\item[(iv)]
Let $r \in [1, \infty)$. Then for any $t>0$ and $f \in L^1(\R^d, \mu)_b$, it holds
\begin{equation} \label{basicriesz}
\|\overline{T}^{\mu}_t f\|_{L^r(\R^d, \mu)} \leq \|f\|_{L^r(\R^d, \mu)}
\end{equation}
and the unique continuous extensions $(T^{\mu}_t)_{t>0}$ of $(\overline{T}^{\mu}_t)_{t>0}|_{L^1(\R^d, \mu)_b}$ to $L^r(\R^d, \mu)$ form sub-Markovian $C_0$-semigroup of contractions on $L^r(\R^d, \mu)$. Moreover, let $(L^{\mu}_r, D(L^{\mu}_r))$ be the generator 
in $L^r(\R^d, \mu)$ associated with $(T^{\mu}_t)_{t>0}$. If $f \in D(\overline{L}^{\mu}) \cap L^r(\R^d, \mu)$ and $\overline{L}^{\mu} f \in L^r(\R^d, \mu)$, then $f \in D(L^{\mu}_r)$ and $\overline{L}^{\mu} f = L^{\mu}_r f$.

\item[(v)]
Let $r \in [1, \infty)$. Then for any $t>0$ and $f \in L^1(\R^d, \mu)_b$, it holds
\begin{equation}
\|\overline{T}'^{, \mu}_t f\|_{L^r(\R^d, \mu)} \leq \|f\|_{L^r(\R^d, \mu)}
\end{equation}
and the unique continuous extensions $(T'^{, \mu}_t)_{t>0}$ of $(\overline{T}'^{, \mu}_t)_{t>0}|_{L^1(\R^d, \mu)_b}$ to $L^r(\R^d, \mu)$ form sub-Markovian $C_0$-semigroup of contractions on $L^r(\R^d, \mu)$. Moreover, let $(L'^{, \mu}_r, D(L'^{, \mu}_r))$ be the generator 
in $L^r(\R^d, \mu)$ associated with $(T'^{, \mu}_t)_{t>0}$. If $f \in D(\overline{L}'^{, \mu}) \cap L^r(\R^d, \mu)$ and $\overline{L}'^{, \mu} f \in L^r(\R^d, \mu)$, then $f \in D(L'^{, \mu}_r)$ and $\overline{L}'^{, \mu} f = L'^{, \mu}_r f$.

\item[(vi)]
For $t>0$ and $f \in L^{\infty}(\R^d, \mu)$ with $f \geq 0$, define
\begin{equation} \label{deflinf}
T^{\mu}_t f:= \lim_{n \rightarrow \infty} T^{\mu}_t f_n,
\end{equation}
where $(f_n)_{n \geq 1}$ is an increasing sequence of non-negative functions in $L^1(\R^d, \mu)_b$ which converges to $f$ \,$\mu$-a.e. Then \eqref{deflinf} is well-defined regardless of the choice of \ $(f_n)_{n \geq 1}$. For $t>0$ and $f \in L^{\infty}(\R^d, \mu)$, define
$$
T^{\mu}_t f := T^{\mu}_t f^+ -T^{\mu}_t f^-.
$$
Then, $(T^{\mu}_t)_{t>0}$ is a sub-Markovian semigroup on $L^{\infty}(\R^d, \mu)$ and
$$
T^{\mu}_t f=\overline{T}^{\mu}_t f, \quad \text{ for all $f \in L^{1}(\R^d, \mu)_b$}.
$$

\item[(vii)]
For $t>0$ and $f \in L^{\infty}(\R^d, \mu)$ with $f \geq 0$, define
\begin{equation} \label{deflinf2}
T'^{, \mu}_t f:= \lim_{n \rightarrow \infty} T'^{, \mu}_t f_n,
\end{equation}
where $(f_n)_{n \geq 1}$ is an increasing sequence of non-negative functions in $L^1(\R^d, \mu)_b$ which converges to $f$ \,$\mu$-a.e. Then \eqref{deflinf2} is well-defined regardless of the choice of \ $(f_n)_{n \geq 1}$. For $t>0$ and $f \in L^{\infty}(\R^d, \mu)$, define
$$
T'^{, \mu}_t f := T'^{, \mu}_t f^+ -T'^{, \mu}_t f^-.
$$
Then, $(T'^{, \mu}_t)_{t>0}$ is a sub-Markovian semigroup on $L^{\infty}(\R^d, \mu)$ and
$$
T'^{, \mu}_t f=\overline{T}'^{, \mu}_t f, \quad \text{ for all $f \in L^{1}(\R^d, \mu)_b$}.
$$
\end{itemize}
\end{prop}
\begin{proof}
(i) Let $f \in C_0^{\infty}(\R^d)$ be given. Then, it follows from the integration by parts that
$$
\int_{\R^d} Lf d\mu = 0, \quad \text{ for all $f \in C_0^{\infty}(\R^d)$}.
$$
Thus, the assertion follows by \cite[Theorem 1.5]{St99}.\\
(ii) Similarly to (i), the assertion follows. \\
(iii) The assertion follows from \cite[Remark 1.7]{St99}.\\
(iv) By Riesz-Thorin interpolation, \eqref{basicriesz} holds. The rest follows by \cite[Lemma 1.11]{Eb}. \\
(v) Similarly to (iii), the assertion follows. \\
(vi) Let $t>0$ and  $f \in L^{\infty}(\R^d, \mu)$ with $f \geq 0$. Choose an increasing sequence of non-negative functions $(f_n)_{n \geq 1}$ in $L^1(\R^d, \mu)_b$ which converges to $f$ \,$\mu$-a.e. Then by the monotone convergence,
$$
\lim_{n \rightarrow \infty } T^{\mu}_t f_n \; \text{ exists $\mu$-a.e.} \text{ and }\lim_{n \rightarrow \infty } T^{\mu}_t f_n \leq \|f\|_{L^{\infty}(\R^d, \mu)} \text{ $\mu$-a.e.} 
$$
Let  $(g_n)_{n \geq 1}$ be an increasing sequence of non-negative functions in $L^1(\R^d, \mu)_b$ which converges to $f$ \,$\mu$-a.e. 
Then
$$
\lim_{n \rightarrow \infty } T^{\mu}_t g_n \; \text{ exists $\mu$-a.e.}  \text{ and }\lim_{n \rightarrow \infty } T^{\mu}_t g_n \leq \|f\|_{L^{\infty}(\R^d, \mu)} \text{ $\mu$-a.e.} 
$$
Let $\varphi \in C_0^{\infty}(\R^d)$. Then, by Lebesgue's theorem and (iii)
\begin{align*}
&\int_{\R^d} \left(\lim_{n \rightarrow \infty} T^{\mu}_t f_n\right) \cdot \varphi \,d\mu=
\lim_{n \rightarrow \infty} \int_{\R^d} T^{\mu}_t f_n \cdot \varphi \,d\mu =\lim_{n \rightarrow \infty} \int_{\R^d} \overline{T}^{\mu}_t f_n \cdot \varphi \,d\mu \\
&\quad =\lim_{n \rightarrow \infty} \int_{\R^d} f_n \cdot \overline{T}'^{, \mu}_t  \varphi \,d\mu= \int_{\R^d} f  \cdot \overline{T}'^{, \mu}_t  \varphi \,d\mu.
\end{align*}
Likewise, 
$$
\int_{\R^d} \left(\lim_{n \rightarrow \infty} T^{\mu}_t g_n\right) \cdot \varphi \,d\mu= \int_{\R^d} f  \cdot \overline{T}'^{, \mu}_t  \varphi \,d\mu,
$$
hence
$$
\lim_{n \rightarrow \infty } T^{\mu}_t f_n=\lim_{n \rightarrow \infty } T^{\mu}_t g_n \;\text{ $\mu$-a.e. }
$$
Thus, \eqref{deflinf} is well-defined. The semigroup property of $(T^{\mu}_t)_{t>0}$ holds since for any $t, s>0$
$$
T_t T_s f = T_t(\lim_{n \rightarrow \infty} T_s f_n) = \lim_{n \rightarrow \infty}T_t T_s f_n =\lim_{n \rightarrow \infty} T_{t+s}f_n = T_{t+s} f, \;\;\text{$\mu$-a.e.}
$$
(vii) Similarly to (vi), the assertion follows.
\end{proof}

\begin{rem} \label{defnofsemi} 
From now on, for each $r \in [1, \infty)$ we denote by $(T^{\mu}_t)_{t>0}$ (resp. $(T’^{, \mu}_t)_{t>0}$) the sub-Markovian $C_0$-semigroup of contractions 
on $L^r(\R^d, \mu)$ as in Proposition \ref{basicprop}(iii), (resp. Proposition \ref{basicprop}(iv)) and denote again by $(T^{\mu}_t)_{t>0}$ (resp. $(T’^{, \mu}_t)_{t>0}$) the sub-Markovian semigroup on $L^{\infty}(\R^d, \mu)$  as in Proposition \ref{basicprop}(v), (resp. Proposition \ref{basicprop}(vi)). Then for each $t>0$ and $f \in L^{r_1}(\R^d, \mu) \cap L^{r_2}(\R^d, \mu)$ for some $r_1, r_2 \in [1, \infty]$, the consistency of $T^{\mu}_t f$  (resp. $T’^{, \mu}_t f$) follows by Proposition \ref{basicprop}(iii), (v) (resp. Proposition \ref{basicprop}(iv), (vi)).
Different from $(\overline{T}^{\mu}_t)_{t>0}$ (resp. $(\overline{T}’^{, \mu}_t)_{t>0}$) only acting on $L^1(\R^d, \mu)$, $(T^{\mu}_t)_{t>0}$ (resp. $(T’^{, \mu}_t)_{t>0}$) acts on $\cup_{s \in [1, \infty]} L^s(\R^d, \mu)$. In particular, if $(T^{\mu}_t)_{t>0}$ (resp. $(T’^{, \mu}_t)_{t>0}$) is considered as a sub-Markovian semigroup on $L^{s}(\R^d, \mu)$ for some $s \in [1, \infty]$, then
$$
\overline{T}^{\mu}_t f= T^{\mu}_t f \quad \text{(resp. $\overline{T}'^{, \mu}f =T'^{, \mu}_t f$)}, \quad \text{ for any $f \in L^1(\R^d, \mu) \cap L^s(\R^d, \mu)$ and $t>0$}.
$$
Thus, unless otherwise stated, we will exclusively use $(T^{\mu}_t)_{t>0}$ (resp. $(T'^{, \mu}_t)_{t>0}$)
instead of $(\overline{T}^{\mu}_t)_{t>0}$ (resp. $(\overline{T}'^{, \mu}_t)_{t>0}$).
\end{rem}
In this section for the convenience of the readers, we review  the definitions of infinitesimally invariant measures, invariant measures, conservativeness, $L^r$-uniqueness, recurrence and transience and some basic results about their relations as presented in \cite{LT21in}. 
\begin{defn} \label{conserdef}
\begin{itemize}
\item[(i)]
A positive, locally finite measure $\widehat{\mu}$ defined on $(\R^d, \mathcal{B}(\R^d))$ satisfying $Lf \in L^1(\R^d, \widehat{\mu})$ for all $f \in C_0^{\infty}(\R^d)$ is called an infinitesimally invariant measure for $(L,C_0^{\infty}(\R^d))$, if
$$
\int_{\R^d} Lf d\widehat{\mu} = 0, \qquad \forall f \in C_0^{\infty}(\R^d).
$$
\item[(ii)] A positive, locally finite measure $\widetilde{\mu}$ defined on $(\R^d, \mathcal{B}(\R^d))$ with $\widetilde{\mu} \ll  \mu$, and $Lf \in L^1(\R^d, \widetilde{\mu})$ for all $f \in C_0^{\infty}(\R^d)$, is called an invariant measure for $(T^{\mu}_t)_{t>0}$ (acting on $L^{\infty}(\R^d, \mu)$), if
$$
\int_{\R^d} T^{\mu}_t 1_A d \widetilde{\mu}  = \widetilde{\mu}(A), \qquad \forall A \in \mathcal{B}(\R^d),\; t>0.
$$
\item[(iii)]
$(T^{\mu}_t)_{t>0}$ (acting on $L^{\infty}(\R^d, \mu)$) is called conservative, if
$$
T^{\mu}_t 1_{\mathbb{R}^d} = 1 \; \text{ $\mu$-a.e.\;\; for one (and hence all) $t>0$. }
$$
$(T^{\mu}_t)_{t>0}$ is called non-conservative, if $(T^{\mu}_t)_{t>0}$ is not conservative.
\end{itemize}
\item[(iv)] We say that $\mu$ is the unique infinitesimally invariant measure (resp. invariant measure) for $(L, C_0^{\infty}(\R^d))$ (resp. $(T^{\mu}_t)_{t>0}$) if whenever $\widetilde{\mu}$ is an infinitesimally invariant measure (resp. invariant measure) for $(L, C_0^{\infty}(\R^d))$ (resp. $(T^{\mu}_t)_{t>0}$), then there exists a constant $c>0$ such that $\widetilde{\mu}=c \mu$.
\item[(v)] $(L, C_0^{\infty}(\R^d))$ is said to be $L^r(\R^d, \mu)$-unique, $r \in [1, p]$, if $(T^{\mu}_t)_{t>0}$ is the only $C_0$-semigroup on $L^r(\R^d, \mu)$, whose generator extends $(L, C_0^{\infty}(\R^d))$.
\item[(vi)] Let $(G^{\mu}_{\alpha})_{\alpha>0}$ and $(G'^{, \mu}_{\alpha})_{\alpha>0}$ be the resolvents associated to $(T^{\mu}_t)_{t>0}$ and $(T'^{, \mu}_t)_{t>0}$
on $L^r(\R^d, \mu)$, $r \in [1, \infty]$, respectively, i.e.
$$
G^{\mu}_{\alpha} f := \int_0^{\infty} e^{-\alpha t} T^{\mu}_t f dt, \;\;\;\; G'^{,\mu}_{\alpha} f := \int_0^{\infty} e^{-\alpha t} T'^{, \mu}_t f dt, \; \quad f \in L^r(\R^d, \mu), \,\alpha>0.
$$
For each $f \in L^1(\R^d, \mu)$, $f  \geq 0$ $\mu$-a.e, let 
$$
G^{\mu}f:= \lim_{N \rightarrow \infty} \int_0^N T^{\mu}_t f \, d\mu = \lim_{\alpha \rightarrow 0+ } \int_{0}^{\infty} e^{-\alpha t} T^{\mu}_t f \, d\mu =  \lim_{\alpha \rightarrow 0+} G^{\mu}_{\alpha} f, \quad \mu\text{-a.e.}
$$
$(T^{\mu}_t)_{t>0}$ is said to be recurrent, if for any $f \in L^1(\R^d, \mu)$ with $f \geq 0$ $\mu$-a.e. we have
$$
G^{\mu}f  \in \{0, \infty\}, \;\; \mu \text{-a.e.},
$$
i.e. $\mu(\{0< G^{\mu}f < \infty\})=0$.\\
$(T^{\mu}_t)_{t>0}$ is said to be transient, if there exists $g \in L^1(\R^d, \mu)$ with $g> 0$  $\mu$-a.e. such that
$$
G^{\mu}g <\infty, \;\; \mu \text{-a.e.}
$$
\item[(vii)]
Given a sub-Markovian $C_0$-semigroup of contractions $(S_t)_{t>0}$ on $L^r(\R^d, \mu)$ for some $r \in [1, \infty)$ that satisfies $(S_t)_{t>0} = (T^{\mu}_t)_{t>0}$ on $L^r(\R^d, \mu)$, $(S_t)_{t>0}$ is called recurrent (resp. transient), if $(T^{\mu}_t)_{t>0}$ is recurrent (resp. transient).
\end{defn}

\begin{prop} \label{prop2.2}
\begin{itemize}
\item[(i)]
Let $r \in [1, p]$ and $\alpha>0$. For $h \in L^{q}(\R^d, \mu)$ with $q \in (1, \infty]$ and $\frac{1}{r}+\frac{1}{q}=1$,
$$
\int_{\R^d} (\alpha-L) u \cdot h d\mu =0, \quad \forall u \in C_0^{\infty}(\R^d)
$$
implies that $h=0$ $\mu$-a.e. if and only if $(L, C_0^{\infty}(\R^d))$ is $L^r(\R^d, \mu)$-unique.

\item[(ii)]
$\mu$ is an invariant measure  for $(T^{\mu}_t)_{t>0}$ if and only if $(L, C_0^{\infty}(\R^d))$ is $L^1(\R^d, \mu)$-unique.

\item[(iii)]
$\mu$ is an invariant measure for $(T^{\mu}_t)_{t>0}$ if and only if $(T'^{, \mu}_t)_{t>0}$ is conservative.

\item[(iv)]
Assume $\mu$ is finite. Then $\mu$ is an invariant measure for $(T^{\mu}_t)_{t>0}$ if and only if $(T^{\mu}_t)_{t>0}$ is conservative. Moreover, $(T^{\mu}_t)_{t>0}$ is recurrent if and only if $(T^{\mu}_t)_{t>0}$ is conservative.

\item[(v)]
$(T^{\mu}_t)_{t>0}$ is either recurrent or transient.
\end{itemize}
\end{prop}
\begin{proof}
(i) It follows from \cite[Lemma 1.11]{Eb} that $(L, C_0^{\infty}(\R^d))$ is $L^r(\R^d, \mu)$-unique if and only if $(L^{\mu}_r, D(L^{\mu}_r))$ is the 
closure of $(L, C_0^{\infty}(\R^d))$ on $L^r(\R^d, \mu)$. Using the fact that $G^{\mu}_{\alpha}=(\alpha-L^{\mu}_r)^{-1}$ is a bounded linear operator on $L^r(\R^d, \mu)$ for all $\alpha>0$, the latter above is equivalent to the fact that
\begin{equation} \label{densebe}
(\alpha-L)(C_0^{\infty}(\R^d)) \subset L^r(\R^d, \mu)  \text{ dense with respect to } \| \cdot \|_{L^r(\R^d, \mu)}  \text{ for any } \alpha>0.
\end{equation}
As a consequence of the Hahn-Banach theorem (\cite[Proposition 1.9]{BRE}) and the Riesz representation theorem (\cite[Theorems 4.11, 4.14]{BRE}), \eqref{densebe} is equivalent to the former of (i) (cf. \cite[Remark 3.4]{LT21in}). \\
(ii)--(iv) The assertions follow from \cite[Corollary 2.2]{St99}, \cite[Remark 2.13(i)]{LT20} and \cite[Remark 2.13(ii)]{LT20}, respectively. \\
(v) The assertion follows from \cite[Theorem 3.38(i)]{LT20} (cf. \cite[Proposition 4.9]{LT18}).
\end{proof}

\section{Conservativeness and existence of an invariant measure} \label{invnewassum}
Throughout this section, we deal with $(L, C_0^{\infty}(\R^d))$ satisfying {\bf (S1)} as in Section \ref{intro} (Introduction) with $\mu=\rho dx$ as  infinitesimally invariant measure for $(L, C_0^{\infty}(\R^d))$. If $\mu$ is finite and $(T^{\mu}_t)_{t>0}$ is conservative, then $(T^{\mu}_t)_{t>0}$ is recurrent by Proposition \ref{prop2.2}(iv), hence \cite[Theorem 3.15 and Theorem 3.17]{LT21in} apply and can be used to prove uniqueness of the infinitesimally invariant measures and invariant measures. If $\mu$ is infinite and $(T^{\mu}_t)_{t>0}$ is conservative, then \cite[Proposition 3.20]{LT21in} applies and can be used to prove the non-existence of finite infinitesimally invariant measures. In any case, conservativeness of $(T^{\mu}_t)_{t>0}$ under the assumption of {\bf (S1)} can be used to investigate finite infinitesimally invariant measures for $(L, C_0^{\infty}(\R^d))$.\\
\text{} \\
Let $B=(b_{ij})_{1 \leq i,j \leq d}$ be a matrix of  functions on $\R^d$. The real-valued functions $\Phi_B$, $|B|_{\infty}$ on $\R^d$ are defined as
\begin{equation} \label{defnmatrixnorm}
\Phi_B(x) := \inf_{\xi \in \R^d \setminus \{0\}} \frac{\langle B(x) \xi, \xi \rangle}{\|\xi\|^2}, \;\; \;\;
|B|_{\infty}(x):=\max_{1 \leq i,j \leq d} |b_{ij}(x)|,
\quad \; x \in \R^d.
\end{equation}
Since $A=(a_{ij})_{1\leq i,j \leq d}$ is locally uniformly strictly elliptic, it holds 
$\Phi_A(x)>0$ for all $x \in \R^d$.
By the Cauchy-Schwarz inequality, it follows
\begin{equation} \label{diffestima}
\|B(x) \xi \| \leq d |B|_{\infty}(x)\|\xi\|, \quad \forall x,  \xi \in \R^d.
\end{equation}
\text{}\\
The idea of the proof of Theorem \ref{theo3:1}(i) is inspired by the proof of \cite[Lemma 5.4]{RoShTr}. In case where $\rho(x)>0$ for all $x \in \R^d$ in \cite{RoShTr}, Theorem \ref{theo3:1}(i) generalizes \cite[Lemma 5.4]{RoShTr} in that more general coefficients $A$, $C$ and $\overline{\mathbf{B}}$ are covered in Theorem \ref{theo3:1}(i) to obtain the conservativeness of $(T^{\mu}_t)_{t>0}$ and $(T'^{, \mu}_t)_{t>0}$.

\begin{theo} \label{theo3:1}
Assume {\bf (S1)} holds. Then the following (i)--(ii) hold.
\begin{itemize}
\item[(i)]
Assume that there exist constants $c_1, c_2>0$, $\alpha \in [0,1)$, $\beta \in [0, \infty)$, $C_1>0$ and $N_0 \in \N$ such that 
\begin{equation*}
\mu(B_r) \leq c_1r^{\beta}+c_2 \;\; \text{ for all $r>0$}
\end{equation*}
and for a.e. $x\in \R^d\setminus \overline{B}_{N_0}$,
$$
|A|_{\infty}(x)+|C|^2_{\infty}(x) \Phi_A^{-1}(x)+\| \overline{\mathbf{B}}(x) \|^2 \leq C_1 \|x\|^{2\alpha}
$$
where {\color{blue} }
$\Phi_A$ is defined as in \eqref{defnmatrixnorm}. 
Then, $(T^{\mu}_t)_{t>0}$ and $(T'^{, \mu}_t)_{t>0}$ are conservative (or equivalently, by Proposition \ref{prop2.2} (ii), (iii), $(L, C_0^{\infty}(\R^d))$ and $(L'^{, \mu}, C_0^{\infty}(\R^d))$ are $L^1(\R^d, \mu)$-unique, and $\mu$ is an invariant measure for $(T^{\mu}_t)_{t>0}$ and $(T'^{, \mu}_t)_{t>0}$).
\item[(ii)]
Assume that
$$
\frac{|A|_{\infty}(x)+|C|^2_{\infty}(x) \Phi_A^{-1}(x)}{1+\|x\|^2}  +\frac{\|\overline{\mathbf{B}}(x)\|}{1+\|x\|} \in L^1(\R^d, \mu).
$$
Then, the same result as in (i) holds.
\end{itemize}
\end{theo}
\begin{proof}
(i)
Let $h \in L^{\infty}(\R^d, \mu)$ be such that
\begin{equation} \label{basicinveq}
\int_{\R^d} (1-L)u \cdot h d\mu = 0, \qquad \forall u \in C_0^{\infty}(\R^d).
\end{equation}
In order to show $(L, C_0^{\infty}(\R^d))$ is $L^1(\R^d, \mu)$-unique, it is enough to show that $h=0$, $\mu$-a.e. by Proposition \ref{prop2.2}(i). Note that $h \in H^{1,p}_{loc}(\R^d) \cap C(\R^d)$ by \cite[Corollary 2.10]{BKR2}, hence using integration by parts in \eqref{basicinveq} and an approximation,
\begin{equation} \label{invintegraby1}
\frac12 \int_{\R^d} \langle (A+C) \nabla u, \nabla h \rangle d\mu - \int_{\R^d} \langle \overline{\mathbf{B}}, \nabla u \rangle h d\mu + \int_{\R^d} uh d\mu = 0, \qquad \forall u \in H^{1,2}(\R^d)_0.
\end{equation}
Given $n \in \N$, let $u_n \in H^{1,2}(\R^d)_{0}$ be such that $1_{B_n} \leq u_n \leq 1_{B_{2n}}$ and $\|\nabla u_n\|_{\infty} \leq \frac{K}{n}$ for some constant $K>0$ which is independent of $n \in \N$. Note that since $h u_n, h u_n^2 \in  H^{1,2}(\R^d)_{0,b}$, it follows that
$$
 \int_{\R^d} \langle \overline{\mathbf{B}}, \nabla (h u_n) \rangle h u_n d\mu=0,
$$ 
hence
\begin{equation*} \label{divfreebu}
\int_{\R^d} \langle \overline{\mathbf{B}}, \nabla (hu_n^2) \rangle h d\mu =\int_{\R^d} \langle \overline{\mathbf{B}}, \nabla u_n \rangle h^2 u_n d\mu.
\end{equation*}
Take $n \in \N$ with $n \geq N_0$. Since $h u_n^2 \in H^{1,2}_0(\R^d, \mu)_0$, using \eqref{invintegraby1} and the Cauchy-Schwarz inequality,
\begin{eqnarray*}
\int_{B_{n}} h^2 d\mu &\leq& \int_{\R^d} (hu_n^2) h d\mu = -\frac12 \int_{\R^d} \langle (A+C) \nabla (hu_n^2), \nabla h \rangle d\mu +\int_{\R^d} \langle \overline{\mathbf{B}}, \nabla (hu_n^2) \rangle h d\mu \\
&=& -\frac12 \int_{\R^d} \langle \nabla (hu_n^2), (A+C^T)\nabla h \rangle d\mu+\int_{\R^d} \langle \overline{\mathbf{B}}, \nabla u_n \rangle h^2 u_n d\mu \\
 &=& -\frac{1}{2} \int_{B_{2n}} \langle u_n^2 (A+C^T) \nabla h, \nabla h\rangle d\mu - \int_{B_{2n}} \langle (A+C^T) \nabla h, h u_n \nabla u_n  \rangle d\mu +\int_{B_{2n}} \langle \overline{\mathbf{B}}, \nabla u_n \rangle h^2 u_n d\mu \\
&\leq& -\frac{1}{2} \int_{B_{2n}} \langle A \nabla h, \nabla h \rangle u_n^2 d\mu +  \left(\int_{B_{2n}} \langle A \nabla h, \nabla h \rangle u_n^2 d\mu \right)^{1/2} \left( \int_{B_{2n}} \langle A \nabla u_n, \nabla u_n \rangle h^2 d\mu \right)^{1/2} \\
&& \;\;\; +\int_{B_{2n}} |\langle C^T \nabla h, hu_n \nabla u_n\rangle|  d\mu +\int_{B_{2n}} \|\overline{\mathbf{B}}\| \|\nabla u_n \| h^2 u_n d\mu 
\end{eqnarray*}
Using \eqref{diffestima}, the Cauchy-Schwarz inequality and the Young inequality $|ab|\leq \frac{a^2}{8}+2b^2$,
\begin{eqnarray}
\int_{B_{n}} h^2 d\mu  &\leq&  -\frac{3}{8} \int_{B_{2n}} \langle A \nabla h, \nabla h \rangle u_n^2 d\mu +2  \int_{B_{2n}} d|A|_{\infty} \|\nabla u_n\|^2h^2 d\mu  \nonumber \\
&& \quad +\int_{B_{2n}}  d|C|_{\infty} |h|u_n  \|\nabla h\| \| \nabla u_n\| d\mu +\int_{B_{2n}} \| \overline{\mathbf{B}}\| \|\nabla u_n \|h^2 u_n d\mu  \nonumber \\ 
&=& -\frac{1}{4}\int_{B_{2n}} \Phi_A \|\nabla h \|^2 u_n^2 d\mu +2  \int_{B_{2n} } d|A|_{\infty} \|\nabla u_n\|^2 h^2 d\mu \nonumber \\
&&\quad +2\int_{B_{2n}} \frac{d^2|C|_{\infty}^2}{\Phi_A}  \|\nabla u_n \|^2  h^2 d\mu +\int_{B_{2n}} \|\overline{\mathbf{B}}\| \|\nabla u_n \| h^2 u_n d\mu. \label{imporequin}
\end{eqnarray}
Therefore,
\begin{eqnarray}
\int_{B_{n}} h^2d\mu& \leq& \left(2^{2\alpha+1} C_1 K^2 (d+d^2)\frac{1}{n^{2-2\alpha}} + 2^{\alpha} \sqrt{C_1}  K \frac{1}{n^{1-\alpha}} \right) \int_{B_{2n} } h^2 d \mu  \nonumber \\
&\leq& \underbrace{\left(2^{2\alpha+1} C_1 K^2 (d+d^2) + 2^{\alpha} \sqrt{C_1} K \right)  }_{=:C_2}\frac{1}{n^{1-\alpha}} \int_{B_{2n}} h^2 d\mu. \nonumber
\end{eqnarray}
Let $\phi(n):=\int_{B_n} h ^2 d\mu$. Then for any $k \in \N$ with $k > \frac{\beta}{1-\alpha}$,
$$
\phi(n) \leq \left(\frac{C_2}{n^{1-\alpha}}\right)^k \phi(2^k n) \leq \left(\frac{C_2}{n^{1-\alpha}}\right)^k \|h\|^2_{L^{\infty}(\R^d)}\, (c_1 2^{k \beta} n^{\beta}+c_2)= \frac{c_1C_2^k 2^{k \beta} \|h\|^2_{L^{\infty}(\R^d)}}{n^{(1-\alpha)k-\beta}} +\frac{c_2C_2^k \|h\|^2_{L^{\infty}(\R^d)}}{n^{(1-\alpha)k}}. 
$$
Thus, $\int_{\R^d} h^2 d\mu = \lim_{n \rightarrow \infty} \phi(n)=0$, which yields $h=0$, $\mu$-a.e. Therefore, $(L, C_0^{\infty}(\R^d))$ is $L^1(\R^d, \mu)$-unique.  By replacing $\overline{\mathbf{B}}$, $C$ with $-\overline{\mathbf{B}}$, $C^T$, respectively, it also follows that $(L'^{, \mu}, C_0^{\infty}(\R^d))$ is $L^1(\R^d, \mu)$-unique. The remaining assertions follow from Proposition \ref{prop2.2}(ii), (iii). \\[5pt]
(ii) 
Note that \eqref{imporequin} holds under the assumption of {\bf (S1)}.
For $x \in B_{2n} \setminus B_n$, we have 
$$
\frac{\|x\|+1}{4} \leq \frac{\|x\|}{2}<n \; \text{ and } \; \frac{\|x\|^2+1}{8}\leq \frac{\|x\|^2}{4}< n^2.
$$
Thus since $\nabla u_n=0$ on $B_n$, \eqref{imporequin} implies that
\begin{eqnarray*}
\int_{B_{n}} h^2 d\mu  &\leq& -\frac{1}{4} \int_{B_{2n}} \Phi_A \|\nabla h \|^2 u_n^2  d\mu 	+ 2dK^2 \int_{B_{2n} \setminus B_n} \frac{|A|_{\infty}}{n^2}  h^2 d\mu  + 2d^2K^2 \int_{B_{2n} \setminus B_n} \frac{|C|^2_{\infty} \cdot \Phi_A^{-1}}{n^2}  h^2 d\mu \\
&& \;\; +  K\int_{B_{2n} \setminus \overline{B}_n}  \frac{\| \overline{\mathbf{B}}\|}{n} h^2  u_n d\mu \leq 16d K^2 \|h\|^2_{L^\infty(\R^d)} \int_{B_{2n} \setminus B_n} \frac{|A|_{\infty}(x)}{\|x\|^2+1} d\mu \\
&&    \;\;+ 16d^2K^2\|h\|^2_{L^\infty(\R^d)} \int_{B_{2n} \setminus B_n} \frac{|C|^2_{\infty}(x) \cdot \Phi_A^{-1}(x)}{\|x\|^2+1} d\mu +  4K\|h\|^2_{L^\infty(\R^d)} \int_{B_{2n} \setminus B_n} \frac{\| \overline{\mathbf{B}}(x)\|}{\|x\|+1} d\mu. \\
\end{eqnarray*}
Since the assumptions of (ii) imply
$$ 
\lim_{n \rightarrow \infty} \int_{B_{4n}\setminus B_{2n}} \frac{|A|_{\infty}(x)}{\|x\|^2+1} d\mu =\lim_{n \rightarrow \infty}  \int_{B_{4n}\setminus B_{2n}} \frac{|C|^2_{\infty}(x) \cdot \Phi_A^{-1}(x)}{\|x\|^2+1} d\mu=\lim_{n \rightarrow \infty} \int_{B_{4n}\setminus B_{2n}}\frac{\|\overline{\mathbf{B}}(x)\|}{\|x\|+1} d\mu=0,
$$
it follows
$$
\int_{\R^d} h^2 d\mu = \lim_{n \rightarrow \infty} \int_{B_{2n}} h^2 d\mu = 0,
$$
hence $h=0$ $\mu$-a.e. Therefore, remaining assertions follow as in (i).
\end{proof}

\begin{rem}   \label{remark:3.2}
\begin{itemize}
\item[(i)]
In contrast to \cite[Propositon 3.31]{LT20} (originally from \cite[Corollary 15(i)]{GT17}), Theorem \ref{theo3:1}(i)  does not require any growth condition on the partial derivative of $c_{ij}$, $1\leq i,j \leq d$. 
This advantage is illustrated in Example \ref{ex:3.10}(iii). On the other hand, the order of $\alpha$ is not optimal in Theorem \ref{theo3:1}(i) (see (ii) right below).

\item[(ii)]
If $a_{ij}, c_{ij}, \|\mathbf{\overline{B}}\| \in L^1(\R^d, \mu)$ for all $1 \leq i,j \leq d$, then the condition of Theorem \ref{theo3:1}(ii) follows. Hence,  Theorem \ref{theo3:1}(ii) is a partial  (since the local regularity assumptions there are more general than {\bf(S1)} here) generalization of \cite[Proposition 1.10(a)]{St99}. If $\mu$ is finite, then Theorem \ref{theo3:1}(ii) shows that $\alpha$ is allowed to be $1$ in Theorem \ref{theo3:1}(i). If $C=0$, then $\alpha$ is allowed to be $1$ in Theorem \ref{theo3:1}(i) by \cite[Proposition 3.31]{LT20}.

\end{itemize}
\end{rem}
As Theorem \ref{theo3:1}(ii), the following result on $L^r$-uniqueness allows a more general growth condition than Theorem \ref{theo3:1}(i) on the coefficients of $L$ to obtain recurrence in the case when $\mu$ is finite. For its proof, we adapt the method of \cite[Theorem 2.3]{Eb}, but partially improve it by considering a growth condition for $C=(c_{ij})_{1 \leq i,j \leq d}$ and more general local regularity assumptions on $A$ and $\overline{\mathbf{B}}$.

\begin{theo} \label{theo:3.3}
Assume that {\bf (S1)} holds and that there exist constants $K>0$ and $N_0 \in \N$ with $N_0 \geq 2$ such that
$$
\frac{\langle A(x)x, x \rangle}{\|x\|^2} + |C|_{\infty}^2(x) \Phi_A^{-1}(x) + |\langle \overline{\mathbf{B}}(x),x \rangle| \ln\|x\| \leq K \left( \|x\| \cdot  \ln \|x\|  \right)^2,  \quad \; \text{ for a.e. } x \in \R^d \setminus \overline{B}_{N_0}.
$$
Then $(L, C_0^{\infty}(\R^d))$ and $(L'^{, \mu}, C_0^{\infty}(\R^d))$ are $L^r(\R^d, \mu)$-unique for all $r \in (1, 2]$. In particular, if $\mu$ is finite, then $(L, C_0^{\infty}(\R^d))$ and $(L'^{, \mu}, C_0^{\infty}(\R^d))$ are $L^1(\R^d, \mu)$-unique, and $(T^{\mu}_t)_{t>0}$ and $(T_t'^{, \mu})_{t>0}$ are recurrent as well as conservative.
\end{theo}
\begin{proof}
For $r \in (1, 2]$, take  $q \in [2, \infty)$ so that $\frac{1}{r}+\frac{1}{q}=1$. Let $h \in L^q(\R^d, \mu)$ be such that
\begin{equation} \label{basicinveqlp}
\int_{\R^d} (1-L)u \cdot h d\mu = 0, \qquad \forall u \in C_0^{\infty}(\R^d). 
\end{equation}
To show that  $(L, C_0^{\infty}(\R^d))$ is $L^r(\R^d, \mu)$-unique, it is by Proposition \ref{prop2.2}(i) enough to show that $h=0$, $\mu$-a.e. By \cite[Corollary 2.10]{BKR2}, we obtain $h \in H^{1,p}_{loc}(\R^d) \cap C(\R^d)$, hence using integration by parts in \eqref{basicinveqlp} and an approximation, it follows
\begin{equation} \label{invintegrabylp}
\frac12 \int_{\R^d} \langle (A+C) \nabla u, \nabla h \rangle d\mu - \int_{\R^d} \langle \overline{\mathbf{B}}, \nabla u \rangle h d\mu + \int_{\R^d} uh d\mu = 0, \qquad \forall u \in H^{1,2}_{loc}(\R^d)_0.
\end{equation}
Note that by the chain rule, $h|h|^{q-2} \in H^{1,p}_{loc}(\R^d) \cap C(\R^d)$ and 
$$
\nabla(h |h|^{q-2}) = (q-1) |h|^{q-2}\nabla h.
$$
Let $v \in H^{1, \infty}(\R^d)_0$. By substituting $v h |h|^{q-2} \in H^{1,2}_0(\R^d, \mu)_0$ for $u$ in \eqref{invintegrabylp}, it follows
\begin{eqnarray*}
\int_{\R^d} v|h|^q d\mu &=&-\frac{q-1}{2} \int_{\R^d} \langle (A+C) \nabla h, \nabla h \rangle |h|^{q-2} vd\mu  + \frac{q-1}{q} \int_{\R^d} \langle \overline{\mathbf{B}}, \nabla (v|h|^q) \rangle d\mu \\
&& \quad  -\frac12 \int_{\R^d} \langle (A+C)\nabla v, \nabla h \rangle h |h|^{q-2} d\mu +\frac{1}{q} \int_{\R^d} \langle \overline{\mathbf{B}}, \nabla v \rangle |h|^{q} d\mu,
\end{eqnarray*}
where we used that
\begin{eqnarray*}
&&h\nabla (v h|h|^{q-2})  = (\nabla v) |h|^{q} + v (q-1) |h|^{q-2} h\nabla h \\
&&= \frac{q-1}{q} |h|^q \nabla v  +(q-1) v |h|^{q-2} h\nabla h  +\frac{1}{q} |h|^q \nabla v = \frac{q-1}{q} \nabla (v |h|^q) +\frac{1}{q}|h|^q\nabla v.
\end{eqnarray*}
Choosing $v = w^2$, where $w \in H^{1,\infty}(\R^d)_0$ and using the Cauchy-Schwarz and the Young inequality $|a b| \leq \frac{q-1}{8}{a^2} +\frac{2}{q-1}b^2$, we obtain
\begin{eqnarray*}
&& \hspace{-1.5em}\int_{\R^d} w^2|h|^q d\mu + \frac{q-1}{2} \int_{\R^d} \langle A \nabla h, \nabla h \rangle |h|^{q-2} w^2 d\mu  \leq  - \int_{\R^d} \langle (A+C) \nabla w, \nabla h \rangle wh|h|^{q-2} d\mu +\frac{2}{q} \int_{\R^d} \langle \overline{\mathbf{B}}, \nabla w \rangle w |h|^q d\mu \\
&& \quad \leq  \left(\int_{\R^d} \langle A \nabla w, \nabla w \rangle  |h|^{q} d\mu \right)^{1/2} \left(\int_{\R^d} \langle A \nabla h, \nabla h \rangle |h|^{q-2} w^2d\mu   \right)^{1/2}  \\
&& \quad \quad \qquad   -\frac{2}{q} \int_{\R^d} \langle \overline{\mathbf{B}}, \nabla w \rangle w |h|^q d\mu  +  \int_{\R^d} d |C|_{\infty} \|\nabla w \| \|\nabla h \| |h|^{q-1} |w|d\mu \\
&&\quad  \leq \frac{q-1}{8} \int_{\R^d}  \langle A \nabla h, \nabla h \rangle |h|^{q-2}  w^2d\mu   +\frac{2}{q-1}\int_{\R^d} \langle  A \nabla w, \nabla w  \rangle  |h|^{q} d\mu   +\frac{2}{q} \int_{\R^d} |\langle \overline{\mathbf{B}}, \nabla w \rangle| |w| |h|^q  d\mu \\
&&\quad \quad \qquad +\frac{q-1}{8} \int_{\R^d}  \Phi_A \| \nabla h \|^2 |h|^{q-2}  w^2  d\mu  + \frac{2}{q-1}\int_{\R^d} d^2 |C|_{\infty}^2 \Phi_A^{-1} \|\nabla w\|^2  |h|^{q}d\mu.
\end{eqnarray*}
Therefore, 
\begin{eqnarray}
&&\int_{\R^d} w^2 |h|^q d\mu \leq\int_{\R^d} w^2|h|^q d\mu + \frac{q-1}{4} \int_{\R^d}  \langle A \nabla h, \nabla h \rangle |h|^{q-2} w^2 d\mu    \nonumber \\
&\leq&  \frac{2}{q-1}\int_{\R^d} \langle  A \nabla w, \nabla w  \rangle  |h|^{q} d\mu  + \frac{2}{q-1}\int_{\R^d} d^2 |C|_{\infty}^2 \Phi_A^{-1} \|\nabla w\|^2  |h|^q d\mu +\frac{2}{q} \int_{\R^d} |\langle \overline{\mathbf{B}}, \nabla w \rangle| |w| |h|^q  d\mu. \quad \qquad \; \label{fundineuniq}
\end{eqnarray}
Define a function $\ell$ on $\R$ by $\ell(t)  := \ln t$ if $t>1$ and $\ell(t):=0$ if $t \leq 1$. Then $\ell$ is monotonically increasing and Lipschitz continuous on $\R$. Moreover,  $(\ell \circ \ell \circ \ell)(t)=\ln \ln \ln t$ if $t \geq e^e$ and $(\ell \circ \ell \circ \ell)(t)=0$, $t < e^e$. For $k \in \N$ with $k \geq N_0+1$, define
$$
A_k :=  \{ x \in \R^d \mid (\ell \circ \ell \circ \ell)(\|x\|) \leq k \} \;\;\text{ and }\;\; w_k := \min \left(  \Big(k- (\ell \circ \ell \circ \ell)(\|x\|)\Big)^+, 1   \right).
$$
Then $w_k \in H^{1,\infty}(\R^d)_0$ with $\|w_k\|_{L^{\infty}(\R^d)} \leq 1$, 
$w_k(x) =1$ on $A_{k-1}$, $w_k=0$ on $\R^d \setminus A_k$, and $w_k(x)= k-\ln \ln \ln\|x\|$  for all $x \in A_k \setminus A_{k-1}$. Therefore, we obtain $\nabla w_k=0$ on $A_{k-1} \cup ({\R^d} \setminus A_{k})$ and 
$$
\nabla w_k(x)= \left(\frac{1}{\|x\| \cdot  \ln \|x\| \cdot \ln \ln \|x\|} \right) \frac{x}{\|x\|}, \quad \text{ for a.e. $x \in A_k \setminus A_{k-1}$}.
$$
Therefore by replacing $w$ with $w_k$ in \eqref{fundineuniq},
\begin{eqnarray*}
\int_{A_{k-1}} |h|^q d\mu &\leq& \sup_{A_k \setminus A_{k-1}} \left( \frac{\langle A(x)x, x \rangle}{\|x\|^2} \frac{1}{\left(\|x\| \cdot  \ln \|x\| \cdot \ln \ln \|x\| \right)^2}  \right) \cdot \frac{2 \|h\|_{L^q(\R^d, \mu)}}{q-1}  \\
&& \;\; + \sup_{A_k \setminus A_{k-1}}  \left(   \frac{|C|_{\infty}^2(x) \Phi_A^{-1}(x)}{\left(\|x\| \cdot  \ln \|x\| \cdot \ln \ln \|x\| \right)^2} \right)  \cdot \frac{2d^2 \|h\|_{L^q(\R^d, \mu)}}{q-1} \\
&& \quad \;\; + \sup_{A_k \setminus A_{k-1}} \left( \frac{|\langle \overline{\mathbf{B}}(x),x \rangle|}{\|x\|} \frac{\|x\| \ln \|x\|}{(\|x\| \cdot  \ln \|x\|)^2 \cdot \ln \ln \|x\| }\right) \cdot  \frac{2\|h\|_{L^q(\R^d, \mu)}}{q} \\
&\leq& \left(\frac{2K}{(e^{k-1})^2}  + \frac{K}{e^{k-1}} \right)\frac{2d^2\|h\|_{L^q(\R^d, \mu)}}{q-1}.
\end{eqnarray*}
Since $A_k \subset A_{k+1}$ for all $k \in \N$ and $\R^d =\cup_{k \in \N} A_k$, we obtain $\int_{\R^d} |h|^q d\mu =\lim_{k \rightarrow \infty}\int_{A_k} |h|^q d\mu=0$, hence $h=0$, $\mu$-a.e. Replacing $C$, $\overline{\mathbf{B}}$ with $C^T$, $-\overline{\mathbf{B}}$, respectively, we also obtain that $(L'^{, \mu}, C_0^{\infty}(\R^d))$ is $L^r(\R^d, \mu)$-unique. The rest of the assertion follows from Proposition \ref{prop2.2}(i)-(v).
\end{proof}

\section{Uniqueness of invariant measures} \label{invasec}
It is shown in \cite[Theorems 3.15, 3.17]{LT21in} that recurrence of $(T^{\mu}_t)_{t>0}$ implies the uniqueness of $\mu$ as infinitesimally invariant measure for $(L, C_0^{\infty}(\R^d))$ and as invariant measure for $(T^{\mu}_t)_{t>0}$ in the sense of Definition \ref{conserdef}(iv). So, for instance, if the growth condition on the coefficients in Theorem \ref{theo:3.3} holds and $\mu$ is finite, then $(T^{\mu}_t)_{t>0}$ is recurrent and $\mu$ is the unique infinitesimally invariant measure for $(L, C_0^{\infty}(\R^d))$ and invariant measure for $(T^{\mu}_t)_{t>0}$. In order to explore further the uniqueness of infinitesimally invariant measures even if $(T^{\mu}_t)_{t>0}$ is transient, we consider the following more restrictive condition than {\bf (S1)}. \\
\text{}\\
{\bf (S2)}: {\it {\bf (S1)} holds with $\overline{\mathbf{B}}=0$, i.e.
\begin{eqnarray*}
Lf &:=& \frac{1}{2\rho}\text{\rm div}\left(\rho (A+C) \nabla f\right)=
\frac12 \text{\rm trace}(A \nabla^2 f) +\langle \beta^{\rho, A+C^T}, \nabla f  \rangle, \quad f \in C_0^{\infty}(\R^d), \\
L'^{, \mu}f &:=& \frac{1}{2\rho}\text{\rm div}\left(\rho (A+C^T) \nabla f\right)=\frac12 \text{\rm trace}(A \nabla^2 f) +\langle \beta^{\rho, A+C}, \nabla f  \rangle, \quad f \in C_0^{\infty}(\R^d).
\end{eqnarray*}
Furthermore, there exist constants $\theta, M>0$ such that
$$
\theta \|\xi\|^2 \leq \langle \rho(x) A(x) \xi, \xi   \rangle \leq M \|\xi\|^2 \;\; \text{ and } \;\; \max_{1 \leq i,j \leq d} | \rho (x)c_{ij}(x)| \leq M, \quad \forall x, \xi \in \R^d.
$$ 
}

\begin{theo} \label{uniqueinfinitesim}
Assume {\bf (S2)} holds. Then $\mu$ is the unique infinitesimally invariant measure for $(L, C_0^{\infty}(\R^d))$ and $(L'^{, \mu}, C_0^{\infty}(\R^d))$.
\end{theo}
\begin{proof}
Let $\widetilde{\mu}$ be an infinitesimally invariant measure for $(L, C_0^{\infty}(\R^d))$. Then by \cite[Corollary 2.10, 2.11]{BKR2}, there exists $\widetilde{\rho} \in H^{1,p}_{loc}(\R^d) \cap C(\R^d)$ such that $\widetilde{\rho}(x)>0$ for all $x \in \R^d$ and $\widetilde{\mu}= \widetilde{\rho} dx$. Let  $u:= \frac{\widetilde{\rho}}{\rho}$. Then $u \in H^{1,p}_{loc}(\R^d) \cap C(\R^d)$ with $u(x)>0$ for all $x \in \R^d$ and
$$
\int_{\R^d} \frac{1}{2}\text{\rm div}\left(\rho (A+C) \nabla \varphi \right) \cdot u \,dx= \int_{\R^d} \rho L\varphi \cdot u \,dx  =  \int_{\R^d} L \varphi\, d\widetilde{\mu} =0, \quad \forall \varphi \in C_0^{\infty}(\R^d).
$$
Using integration by parts, we get
$$
\int_{\R^d} \langle \frac12 \rho (A+C^T) \nabla u, \nabla \varphi \rangle dx =0, \quad \forall \varphi \in C_0^{\infty}(\R^d).
$$
By the Liouville-type theorem \cite[Theorem 9.11 (i)]{F07}, $u$ is constant. This shows the uniqueness of infinitesimally invariant measures for $(L, C_0^{\infty}(\R^d))$. By replacing $C^T$ with $C$ and using the same argument as above, $\mu$ is the unique infinitesimally invariant measure for $(L'^{, \mu}, C_0^{\infty}(\R^d))$.
\end{proof}

\begin{rem}
Assume that {\bf (S2)} holds and let $r \in [2, \infty]$.  Let $(G^{\mu}_{\alpha})_{\alpha>0}$ and $(G'^{, \mu}_{\alpha})_{\alpha>0}$
be the resolvent on $L^r(\R^d, \mu)$
associated with $(T^{\mu}_{t})_{t>0}$ and $(T'^{, \mu}_t)_{t>0}$, respectively (see Definition \ref{conserdef}(vi)). Then we obtain the following result:
\begin{equation} \label{liouviltype}
\text{ If \,$f \in L^r(\R^d, \mu)$ satisfies \ $\alpha G^{\mu}_{\alpha} f=f$ \,for some $\alpha>0$, then $f$ is constant. }
\end{equation}
Indeed, 
let $\alpha>0$ and $g \in L^1(\R^d, \mu)_b$. Then, $G^{\mu}_{\alpha}g \in H^{1,2}_{loc}(\R^d)$ and it follows from \cite[Remark 1.7(ii)]{St99} that for any $\varphi \in C_0^{\infty}(\R^d)$
\begin{eqnarray}
\int_{\R^d} (\rho g)\cdot \varphi dx&=& \int_{\R^d} g \varphi d\mu= \int_{\R^d} g \cdot G'^{, \mu}_{\alpha} (\alpha-L'^{, \mu}) \varphi d\mu =\int_{\R^d} G^{\mu}_{\alpha}g \cdot (\alpha-L'^{, \mu}) \varphi d\mu \nonumber \\
&=&\int_{\R^d} G^{\mu}_{\alpha}g \left(-\frac12 \text{\rm div}\left(\rho (A+C^T) \nabla \varphi \right) \right) dx+\int_{\R^d} \alpha \rho G^{\mu}_{\alpha} g \cdot \varphi dx \nonumber  \\
&=&  \frac12 \int_{\R^d} \Big \langle \rho (A+C) \nabla G_{\alpha}g, \nabla \varphi  \Big \rangle dx+\int_{\R^d} \alpha \rho G^{\mu}_{\alpha} g \cdot \varphi dx. \label{inteidenti}
\end{eqnarray}
By the energy estimate (\cite[Theorem 1.7.4]{BKRS}), for any open balls\, $U_1$ and $U_2$ with $\overline{U}_1 \subset U_2$ there exists a constant $C_1>0$ independent of $g$ such that
\begin{equation} \label{sobolevest}
\|G^{\mu}_{\alpha} g\|_{H^{1,2}(U_1)} \leq C_1 (\|G^{\mu}_{\alpha}g\|_{L^2(U_2)}+\|g\|_{L^2(U_2)}).
\end{equation}
Let $f \in L^r(\R^d, \mu)$ for some $r \in [2, \infty]$. Then there exists a sequence of functions $(f_n)_{n \geq 1} \subset L^1(\R^d, \mu)_b$ such that
$$
\lim_{n \rightarrow \infty}f_n =f, \quad \lim_{n \rightarrow \infty}G^{\mu}_{\alpha}f_n=G^{\mu}_{\alpha}f, \quad \text{$\mu$-a.e.}
$$
Thus, applying the Cauchy sequence argument with Lebesgue's theorem to \eqref{sobolevest}, we have $G_{\alpha}f \in H^{1,2}(U_1)$ and
\begin{equation} \label{approxim}
\lim_{n \rightarrow \infty} G^{\mu}_{\alpha}f_n=G^{\mu}_{\alpha}f \quad \text{ in $H^{1,2}(U_1)$}.
\end{equation}
Since $U_1$ is an arbitrarily chosen open ball, we get $G^{\mu}_{\alpha}f \in H^{1,2}_{loc}(\R^d)$ and it follows from \eqref{approxim} and \eqref{inteidenti} that
\begin{equation} \label{intiden}
\frac12 \int_{\R^d} \Big \langle \rho (A+C) \nabla G_{\alpha}f, \nabla \varphi  \Big \rangle dx+\int_{\R^d} \alpha \rho G^{\mu}_{\alpha} f \cdot \varphi dx =\int_{\R^d} (\rho f)\cdot \varphi dx, \quad \text{ for all $\varphi \in C_0^{\infty}(\R^d)$}.
\end{equation}
Now, assume that $f \in L^r(\R^d, \mu)$ for some $r \in [2, \infty]$ and that $\alpha G^{\mu}_{\alpha} f=f$.
Then, \eqref{intiden} implies that
$$
\frac12 \int_{\R^d} \Big \langle \rho (A+C) \nabla G_{\alpha}f, \nabla \varphi  \Big \rangle dx=0, \quad \text{ for all $\varphi \in C_0^{\infty}(\R^d)$}.
$$
Thus, by the Liouville-type theorem \cite[Theorem 9.11 (i)]{F07}, $G^{\mu}_{\alpha}f$ is constant, hence \eqref{liouviltype} is shown.
Analogously, we obtain the following result:
\begin{equation} \label{liouviltype2}
\text{ If \,$f \in L^r(\R^d, \mu)$ \,satisfies \ $\alpha G'^{, \mu}_{\alpha} f=f$\, for some $\alpha>0$, then $f$ is constant. }
\end{equation}
We refer to \cite{BCR} where various properties for the functions $f$ satisfying that $\alpha G^{\mu}_{\alpha}f=f$ (or $\alpha G'^{, \mu}_{\alpha}f=f$), $\alpha>0$ are systemically explored.
\end{rem}

\begin{lem} \label{rectranap}
Let $\widetilde{A}:=(\widetilde{a}_{ij})_{1 \leq i,j \leq d}$ be a symmetric matrix of functions that is locally uniformly strictly elliptic, with $\widetilde{a}_{ij} \in H^{1,p}_{loc}(\R^d) \cap C(\R^d)$ for all $1 \leq i,j \leq d$, where $p \in (d, \infty)$. Let $\widetilde{\rho}_1 \in H^{1,p}_{loc}(\R^d) \cap C(\R^d)$ with $\widetilde{\rho}_1(x)>0$ for all $x\in \R^d$ and $\widetilde{\mu}_1:=\widetilde{\rho}_1 dx$. Let $(\widetilde{\mathcal{E}}^1, D(\widetilde{\mathcal{E}}^1))$ be the symmetric Dirichlet form defined as the closure of 
$$
\widetilde{\mathcal{E}}^1(f,g):=\frac12 \int_{\R^d} \langle \widetilde{A} \nabla f, \nabla g \rangle d\widetilde{\mu}_1, \quad f, g \in C_0^{\infty}(\R^d)
$$
on $L^2(\R^d, \widetilde{\mu}_1)$ and let $(\widetilde{T}^1_t)_{t>0}$ be the corresponding sub-Markovian $C_0$-semigroup of contractions on $L^2(\R^d, \widetilde{\mu}_1)$. Let $\phi \in H^{1,p}_{loc}(\R^d) \cap C(\R^d) \cap \mathcal{B}_b(\R^d)$ with $\phi(x)>0$ for all $x \in \R^d$ and $\widetilde{\mu}_2:=\phi \widetilde{\rho}_1 dx$. Let $(\widetilde{\mathcal{E}}^2, D(\widetilde{\mathcal{E}}^2))$ be the symmetric Dirichlet form defined as the closure of 
$$
\widetilde{\mathcal{E}}^2(f,g):=\frac12 \int_{\R^d} \langle \frac{1}{\phi} \widetilde{A} \nabla f, \nabla g \rangle  d\widetilde{\mu}_2=\widetilde{\mathcal{E}}^1(f,g), \quad f, g \in C_0^{\infty}(\R^d)
$$
on $L^2(\R^d, \widetilde{\mu}_2)$ and let $(\widetilde{T}^2_t)_{t>0}$ be the corresponding sub-Markovian $C_0$-semigroup of contractions on $L^2(\R^d, \widetilde{\mu}_2)$. Then the following properties {\it (i)-(iii)} are satisfied.
\begin{itemize}
\item[(i)]
If $(\widetilde{T}^1_t)_{t>0}$ is recurrent, then $(\widetilde{T}^2_t)_{t>0}$ is recurrent.
\item[(ii)]
If $(\widetilde{T}^1_t)_{t>0}$ is transient, then $(\widetilde{T}^2_t)_{t>0}$ is transient.
\item[(iii)]
 $(\widetilde{T}^1_t)_{t>0}$ is recurrent (resp, transient) if and only if $(\widetilde{T}^2_t)_{t>0}$ is recurrent (resp, transient).
\end{itemize}
\end{lem}
\begin{proof}
(i) Assume $(\widetilde{T}^1_t)_{t>0}$ is recurrent. Then by \cite[Theorem 1.6.3]{FOT}, there exists a sequence of functions $(u_n)_{n \geq 1} \subset D(\widetilde{\mathcal{E}}^1)$ satisfying $\lim_{n \rightarrow \infty} u_n=1$, $\widetilde{\mu}_1$-a.e. and $\lim_{n \rightarrow \infty} \widetilde{\mathcal{E}}^1(u_n, u_n)=0$. Observe that if $u \in L^2(\R^d, \widetilde{\mu}_1)$, then $u \in L^2(\R^d, \widetilde{\mu}_2)$ and
\begin{equation*}
\|u\|_{L^2(\R^d, \widetilde{\mu}_2)} \leq \|\phi\|^{1/2}_{L^{\infty}(\R^d)} \|u\|_{L^2(\R^d, \widetilde{\mu}_1)},
\end{equation*}
hence $D(\widetilde{\mathcal{E}}^1) \subset D(\widetilde{\mathcal{E}}^2)$. Thus, $(u_n)_{n \geq 1} \subset D(\widetilde{\mathcal{E}}^2)$, $\lim_{n \rightarrow \infty}u_n =1$, $\widetilde{\mu}_2$-a.e. and $\lim_{n \rightarrow \infty}\widetilde{\mathcal{E}}^2(u_n, u_n)=\lim_{n \rightarrow \infty} \widetilde{\mathcal{E}}^1(u_n, u_n)=0$. Therefore, by  \cite[Theorem 1.6.3]{FOT} $(\widetilde{T}^2_t)_{t>0}$ is recurrent. \\
(ii) Assume $(\widetilde{T}^1_t)_{t>0}$ is transient. Then by  \cite[Theorem 1.5.1]{FOT}, there exists a $\widetilde{\mu}_1$-a.e. strictly positive function $g_1 \in L^1(\R^d, \widetilde{\mu}_1)_b$ such that
$$
\int_{\R^d} |u| g_1 d\widetilde{\mu}_1 \leq \sqrt{\widetilde{\mathcal{E}}^1(u,u)}, \quad \forall u \in C_0^{\infty}(\R^d).
$$
Let $g_2:=g_1 \|\phi\|_{L^{\infty}(\R^d)}^{-1}$. Then $g_2 \in L^1(\R^d, \widetilde{\mu}_2)_b$ is $\widetilde{\mu}_2$-a.e. strictly positive and it follows
\begin{eqnarray*}
\int_{\R^d} |u| g_2 d\widetilde{\mu}_2 \leq \int_{\R^d} |u| g_1 d\widetilde{\mu}_1  \leq \sqrt{\widetilde{\mathcal{E}}^1(u,u)} = \sqrt{\widetilde{\mathcal{E}}^2(u,u)}, \quad \forall u \in C_0^{\infty}(\R^d).
\end{eqnarray*}
Since $g_2 \in L^2(\R^d, \widetilde{\mu}_2)$, it follows by approximation that
\begin{eqnarray*}
\int_{\R^d} |u| g_2 d\widetilde{\mu}_2 \leq  \sqrt{\widetilde{\mathcal{E}}^2(u,u)}, \quad \forall u \in D(\widetilde{\mathcal{E}}^2).
\end{eqnarray*}
By  \cite[Theorem 1.5.1]{FOT}, $(\widetilde{T}^2_t)_{t>0}$ is transient. \\
(iii) By \cite[Proposition 3.11]{LT21in} and Proposition \ref{prop2.2}(v), $(\widetilde{T}^1_t)_{t>0}$ is either recurrent or transient, and $(\widetilde{T}^2_t)_{t>0}$ is either recurrent or transient. Hence, the assertion follows from (i) and (ii). 
\end{proof}

\begin{theo} \label{rectrancriter}
Assume that {\bf (S2)} holds and that either 
\begin{equation} \label{rhorinfty}
\rho \in L^{\infty}(\R^d) \,\text{ or }\, \frac{1}{\rho} \in L^{\infty}(\R^d).
\end{equation}
Let  $(\mathcal{E}^{0, \mu}, D(\mathcal{E}^{0, \mu}))$ be the symmetric Dirichlet form defined as the closure of
$$
\mathcal{E}^{0, \mu}(f,g)=\frac12 \int_{\R^d} \langle A \nabla f, \nabla g \rangle d\mu, \quad f, g \in C_0^{\infty}(\R^d)
$$
in $L^2(\R^d, \mu)$ and let $(T^{0, \mu}_t)_{t>0}$ be the corresponding sub-Markovian $C_0$-semigroup of contractions on $L^2(\R^d, \mu)$. Then the following holds:
\begin{itemize}
\item[(i)]
$(T^{0, \mu}_t)_{t>0}$ and $(T^{\mu}_t)_{t>0}$ are recurrent, if $d=2$.
\item[(ii)]
$(T^{0, \mu}_t)_{t>0}$ and $(T^{\mu}_t)_{t>0}$ are transient, if $d \geq 3$. 
\item[(iii)]
If $d \geq 3$ and $\rho \in L^1(\R^d)_b$, then $(T^{\mu}_t)_{t>0}$ is non-conservative and there does not exist an invariant measure for $(T^{\mu}_t)_{t>0}$.
\end{itemize}
\end{theo}
\begin{proof}
For the proof of (i)--(iii), let $\widehat{A}:=\rho A$ and $(\mathcal{B}^0, D(\mathcal{B}^0))$ be the symmetric Dirichlet form defined as the closure of
$$
\mathcal{B}^0(f,g):=\frac12 \int_{\R^d} \langle \widehat{A} \nabla f, \nabla g \rangle dx, \quad f,g \in C_0^{\infty}(\R^d)
$$
in $L^2(\R^d, dx)$ and denote by $(S^0_t)_{t>0}$ the associated sub-Markovian $C_0$-semigroup of contractions on $L^2(\R^d, dx)$. 
Let $d(\cdot, \cdot)$ be an intrinsic metric on $\R^d$ defined by \cite[(1.1)]{Stu2}. Then by \cite[Lemma 2.2]{ST17} there exists a constant $\lambda \in [1, \infty)$ such that
\begin{equation} \label{intequime}
\frac{1}{\sqrt{\lambda}}\|x-y\| \leq d(x,y) \leq \sqrt{\lambda} \|x-y\|, \quad \text{for any $x,y \in \R^d$}.
\end{equation}
(i)
Since $d=2$, 
$$
\int_{1}^{\infty} \frac{r}{dx(B_r)} dr= \frac{1}{dx(B_1)}\int_{1}^{\infty} \frac{r}{r^2} dr=\infty.
$$
Hence, by using \eqref{intequime}, \cite[Corollary 2.9]{Stu2} and \cite[Corollary 4.8(i)]{LT18}, $(S^0_t)_{t>0}$ is recurrent. Thus, by Lemma \ref{rectranap}(iii), $(T^{0, \mu}_t)_{t>0}$ is recurrent. Indeed, if $\rho \in L^{\infty}(\R^d)$, then $\widetilde{A}$, $\widetilde{\rho}_1$, $\phi$, $\widetilde{\mu}_1$, $\widetilde{\mu}_2$ as in Lemma \ref{rectranap} are regarded as $\widehat{A}$, $1$, $\rho$, $dx$, $\mu$, respectively, so that $(\widetilde{T}^1_t)_{t>0}$, $(\widetilde{T}^2_t)_{t>0}$ as in Lemma \ref{rectranap} can be identified as $(S^0_t)_{t>0}$ and $(T^{0, \mu}_t)_{t>0}$, respectively. On the other hand, if $\frac{1}{\rho} \in L^{\infty}(\R^d)$, then $\widetilde{A}$, $\widetilde{\rho}_1$, $\phi$, $\widetilde{\mu}_1$, $\widetilde{\mu}_2$ as in Lemma \ref{rectranap} are regarded as $A$, $\rho$, $\frac{1}{\rho}$, $\mu$, $dx$, so that  $(\widetilde{T}^1_t)_{t>0}$, $(\widetilde{T}^2_t)_{t>0}$ as in Lemma \ref{rectranap} can be identified as $(T^{0, \mu}_t)_{t>0}$ and $(S^0_t)_{t>0}$, respectively.  \\
Let $(\mathcal{E}, D(\mathcal{E}))$  be the Dirichlet form defined as the closure of
$$
\mathcal{E}(f,g):= \frac12 \int_{\R^d} \langle (A+C) \nabla f, \nabla g \rangle d\mu, \quad f,g \in C_0^{\infty}(\R^d)
$$
on $L^2(\R^d, \mu)$. Since for any $f, g \in C_0^{\infty}(\R^d)$ it follows from {\bf (S2)} that
\begin{eqnarray*}
\left|\int_{\R^d} \langle C \nabla f, \nabla g \rangle d\mu \right|  &\leq& dM \left(\int_{\R^d} \|\nabla f\|^2 dx \right)^{1/2} \left(\int_{\R^d} \|\nabla g\|^2 dx \right)^{1/2}  \nonumber \\
&\leq& \frac{2dM}{\theta} \left(\frac12 \int_{\R^d} \langle A \nabla f, \nabla f \rangle d\mu \right)^{1/2} \left(\frac12 \int_{\R^d} \langle A \nabla g, \nabla g \rangle d\mu \right)^{1/2}, \label{strongsecdi}
\end{eqnarray*}
$(\mathcal{E}, D(\mathcal{E}))$ satisfies the strong sector condition defined as in \cite[Chapter I, (2.4)]{MR}. By \cite[Remark 3.12]{LT21in}, $(T^{\mu}_t)_{t>0}$ is associated with $(\mathcal{E}, D(\mathcal{E}))$. Since $(T^{0, \mu}_t)_{t>0}$ is recurrent and $(\mathcal{E}, D(\mathcal{E}))$ satisfies the strong sector condition, $(T^{\mu}_t)_{t>0}$ must be recurrent by \cite[Remark 5(b)]{GT2} and Proposition \ref{prop2.2}(v). \\ \\
(ii) Since $d \geq 3$, 
$$
\int_{1}^{\infty} \frac{r}{dx(B_r)} dr= \frac{1}{dx(B_1)}\int_{1}^{\infty} \frac{r}{r^d} dr<\infty.
$$
Hence,
$(S^0_t)_{t>0}$ is transient by \eqref{intequime}, \cite[Corollary 2.9]{Stu2} and \cite[Corollary 4.8(i)]{LT18}. Using Lemma \ref{rectranap}(iii) and the same way as in (i) above, we obtain $(T^{0, \mu}_t)_{t>0}$ is transient. Therefore, $(T^{\mu}_t)_{t>0}$ is transient by  \cite[Remark 5(b)]{GT2}. \\ \\
(iii) By Proposition \ref{prop2.2}(iv), (v) and Theorem \ref{rectrancriter}(ii) above, we obtain that $\mu$ is not an invariant measure for $(T^{\mu}_t)_{t>0}$ and that $(T^{\mu}_t)_{t>0}$ is non-conservative. Suppose that there exist an invariant measure $\widetilde{\mu}$ for $(T^{\mu}_t)_{t>0}$. Then by \cite[Lemma 3.16]{LT21in}, $\widetilde{\mu}$ is an infinitesimally invariant measure for $(L, C_0^{\infty}(\R^d))$. Thus by Theorem \ref{uniqueinfinitesim}, $\widetilde{\mu}=c \mu$ for some constant $c>0$, which is contradiction since $\mu$ is not an invariant measure for $(T^{\mu}_t)_{t>0}$.
\end{proof}

\begin{exam}
Let $\rho=e^{-\|x\|^2}$, $\mu=\rho dx$, $A=(a_{ij})_{1 \leq i,j \leq d}=e^{\|x\|^2} id$ and $C=(c_{ij})_{1\leq i,j \leq d}=0$. Then {\bf (S2)} holds and $\beta^{\rho, A+C^T}=0$. Then by Theorem \ref{rectrancriter}(i), (ii), $(T^{\mu}_t)_{t>0}$ is recurrent, if $d=2$ and $(T^{\mu}_t)_{t>0}$ is transient and non-conservative, if $d \geq 3$. In particular, by Theorem \ref{rectrancriter}(iii), if $d \geq 3$, then there exists no invariant measure for $(T^{\mu}_t)_{t>0}$. However, by Theorem \ref{uniqueinfinitesim} $\mu$ is the unique infinitesimally invariant measure for $(L, C_0^{\infty}(\R^d))$ for any $d \geq 2$.
\end{exam}

\begin{prop} \label{cor:3.8}
Assume {\bf (S1)} holds, and let $\overline{\mathbf{B}}=0$. Then any one of the following assumptions (i)--(ii) implies that {\bf (S2)} holds and that
$\mu$ is the unique invariant measure for  $(T^{\mu}_t)_{t>0}$ and $(T'^{, \mu}_t)_{t>0}$ as well as $(T^{\mu}_t)_{t>0}$ and $(T'^{, \mu}_t)_{t>0}$ are conservative. 
\begin{itemize}
\item[(i)]
There exist constants $c_1, c_2, \lambda, \Lambda>0$ such that
$$
c_1 \leq \rho(x) \leq c_2, \quad \forall x \in \R^d,
$$
$$
\lambda \|\xi\|^2 \leq \langle A(x) \xi, \xi \rangle \leq \Lambda \|\xi\|^2 \;\; \text{ and }\;\; \max_{1 \leq i,j \leq d} |c_{ij}(x)| \leq \Lambda, \quad \forall x, \xi \in \R^d.
$$

\item[(ii)]
There exist constants $c_1, c_2>0$, $\alpha \in [0,1)$ and $\delta\geq 0$ such that
$$
\frac{c_1}{(1+\|x\|)^{2\alpha}}  \leq \rho(x) \leq c_2(1+\|x\|^{\delta}), \quad \forall x \in \R^d,
$$
$A=(a_{ij})_{1 \leq i,j \leq d}=\frac{1}{\rho} \widetilde{A}$ and $C=(c_{ij})_{1 \leq i,j \leq d}=\frac{1}{\rho} \widetilde{C}$, where
$\widetilde{A}=(\widetilde{a}_{ij})_{1 \leq i,j \leq d}$ is a symmetric matrix of functions with $\widetilde{a}_{ij} \in H^{1,p}_{loc}(\R^d) \cap C(\R^d)$ for all $1 \leq i,j \leq d$, for some $p>d$, such that for some constants $\lambda, \Lambda>0$
$$
\lambda \|\xi\|^2 \leq \langle \widetilde{A}(x) \xi, \xi \rangle \leq \Lambda \|\xi\|^2, \quad \forall x, \xi \in \R^d,
$$
and $\widetilde{C}=(\widetilde{c}_{ij})_{1 \leq i,j \leq d}$ is an anti-symmetric matrix of functions with $\widetilde{c}_{ij} \in H^{1,p}_{loc}(\R^d) \cap C(\R^d)$ such that
$$
\max_{1 \leq i,j \leq d} |\widetilde{c}_{ij}(x)| \leq \Lambda, \quad \forall  x\in \R^d.
$$
\end{itemize}

\end{prop}
\begin{proof}
Assume (i). Then the condition {\bf (S2)} is satisfied and by Theorem \ref{uniqueinfinitesim}, $\mu$ is the unique infinitesimally invariant measure for $(L, C_0^{\infty}(\R^d))$ and $(L'^{, \mu}, C_0^{\infty}(\R^d))$. Moreover, by Theorem \ref{theo3:1}(i), $\mu$ is an invariant measure for $(T^{\mu}_t)_{t>0}$ and $(T'^{, \mu}_t)_{t>0}$. Since by \cite[Lemma 3.16]{LT21in} any invariant measure for $(T^{\mu}_t)_{t>0}$ (resp. for $(T_t'^{, \mu})_{t>0}$ ) is an infinitesimally invariant measure for $(L, C_0^{\infty}(\R^d))$ (resp. for $(L'^{, \mu}, C_0^{\infty}(\R^d))$), the assertion follows. \\
Assume (ii). Then for some contants $c_3, c_4>0$, we get $\mu(B_r) \leq c_3 r^{d+\delta}+c_4$ for all $r>0$ and
$$
\max_{1 \leq i,j \leq d}a_{ij}(x) =\frac{1}{\rho} \max_{1 \leq i,j \leq d}\widetilde{a}_{ij}(x) \leq \frac{\Lambda}{\rho(x)} \leq c_1 \Lambda (1+\|x\|^{2\alpha}), \quad \forall x \in \R^d.
$$
Since
$$
\Phi_A(x)=\inf_{\xi \in \R^d \setminus \{0\}} \frac{\langle A(x) \xi, \xi \rangle}{\|\xi\|^2} \geq \frac{\lambda}{\rho(x)}, \; \quad \max_{1 \leq i,j \leq d} |c_{ij}(x)| \leq \frac{\Lambda}{\rho(x)}, \quad \forall x \in \R^d,
$$
it follows
$$
\left(\max_{1 \leq i,j \leq d}c_{ij}(x)\right)^2\cdot \Phi_A^{-1}(x) \leq  \frac{\Lambda^2}{\lambda \rho(x)} \leq \frac{c_1\Lambda^2}{\lambda} (1+\|x\|^{2 \alpha}), \quad \forall x\in \R^d.
$$
Therefore, by Theorem \ref{theo3:1}(i), $\mu$ is an invariant measure for $(T^{\mu}_t)_{t>0}$ and $(T'^{, \mu}_t)_{t>0}$ 
as well as $(T^{\mu}_t)_{t>0}$ and $(T'^{, \mu}_t)_{t>0}$ are conservative.
The rest of the assertion follows analogously to (i). 
\end{proof}

\begin{rem} \label{rem4.6}
\begin{itemize}
\item[(i)]
If $C=(c_{ij})_{1 \leq i,j \leq d}=0$, then by Remark \ref{remark:3.2}(ii), $\alpha$ of Proposition  \ref{cor:3.8}(ii) is allowed to be $1$ to obtain the same result.
\item[(ii)]
Under the assumption (ii) of Proposition \ref{cor:3.8}, if either $\rho \in L^{\infty}(\R^d)$ or $\frac{1}{\rho} \in L^{\infty}(\R^d)$, then it follows from Theorem \ref{rectrancriter}(i) and (ii) that $(T^{\mu}_t)_{t>0}$ is recurrent if $d=2$ and  transient if $d \geq 3$. Thus Proposition \ref{cor:3.8}(ii) provides a class of $L$ satisfying {\bf (S2)} with $d \geq 3$ for which $(T^{\mu}_t)_{t>0}$ is transient and conservative, but $\mu$ is the unique infinitesimally invariant measure for $(L, C_0^{\infty}(\R^d))$ and $(L'^{, \mu}, C_0^{\infty}(\R^d))$, and the unique invariant measure for $(T^{\mu}_t)_{t>0}$ and $(T'^{, \mu}_t)_{t>0}$ (cf. \cite[Theorems 3.15, 3.17]{LT21in}).
\end{itemize}
\end{rem}
We present explicit examples concerning existence and uniqueness of (infinitesimally) invariant measures using the results of this article. In particular, our results apply to a Brownian motion with singular drift, where existing literature may not be applicable (see the last paragraph of Section \ref{intro}).
\begin{exam} \label{ex:3.10} 
\begin{itemize}
\item[(i)]
Let 
$$
Lf = \frac12 \Delta f +\langle \nabla \phi, \nabla f \rangle, \quad f \in C_0^{\infty}(\R^d),
$$
where $\phi \in H^{1,p}_{loc}(\R^d)$ with $p>d$. Then {\bf (S1)} holds with $A=id$, $\rho = \exp(2\phi)$, $C=(c_{ij})_{1\leq i,j \leq d }=0$, $\overline{\mathbf{B}}=0$ and $\beta^{\rho, A+C^T} = \nabla \phi$, so that $\mu=\exp(2\phi)dx$ is an infinitesimally invariant 
(symmetrizing) measure for $(L, C_0^{\infty}(\R^d))$. If $\mu$ is finite, then $(T^{\mu}_t)_{t>0}$ is recurrent by Theorem \ref{theo:3.3}.  On the other hand if $\mu$ is infinite, then there exists no finite invariant measure for $(T^{\mu}_t)_{t>0}$ by \cite[Proposition 3.20(i)]{LT21in}.  In this case, if for some $N_0 \in \N$
\begin{equation} \label{conscrit1}
\langle \nabla \phi(x), x \rangle \leq M\|x\|^2 \left(\ln\|x\| +1 \right), \quad \text{$\mu$-a.e.}\; x \in \R^d \setminus B_{N_0}
\end{equation}
or
\begin{equation}\label{conscrit2}
\phi(x) \leq M \|x\|^2 \ln(\|x\|+1), \quad \text{$\mu$-a.e. $x \in \R^d \setminus B_{N_0}$},
\end{equation}
then $(T^{\mu}_t)_{t>0}$ is conservative by \cite[Corollary 3.27]{LT20} or \cite[Remark 3.32]{LT20}, respectively, hence it follows from \cite[Corollary 3.23, Theorem 3.52]{LT20} that for each $y \in \mathbb{R}^d$ and probability space $(\widetilde{\Omega}, \widetilde{\mathcal{F}},  \widetilde{\P})$ carrying a $d$-dimensional standard Brownian motion $(\widetilde{W}_t)_{t \geq 0}$
there exists a pathwise unique and strong solution  $(Y^y_t)_{t\geq 0}$ to
\begin{equation} \label{mainsde2} 
Y^y_t =y+\widetilde{W}_s + \int_0^t  \nabla\phi(Y^y_s)ds,  \quad 0 \leq t < \infty.
\end{equation}
Moreover, if $\mu$ is infinite and either \eqref{conscrit1} or \eqref{conscrit2} holds, then there is no finite infinitesimally invariant measure for $(L, C_0^{\infty}(\R^d))$ by \cite[Proposition 3.20(ii)]{LT21in}. \\
Now assume that for some $\alpha_1, \alpha_2 \in \R$ it holds
\begin{equation} \label{uplowcondit}
\alpha_1 \leq \phi(x) \leq \alpha_2, \quad \forall x \in \R^d.
\end{equation}
Then by Theorem \ref{uniqueinfinitesim}, Proposition  \ref{cor:3.8}(i) and Theorem \ref{mainsto}(iii),  $\mu$ is the unique infinitesimally invariant measure for $(L, C_0^{\infty}(\R^d))$ and the unique invariant measure for $(T^{\mu}_t)_{t>0}$  and the unique invariant measure for the family of strong solutions $(Y^y_t)_{t\geq0}$, $y \in \R^d$ to \eqref{mainsde2}  in the sense of Theorem \ref{mainsto}(iii), respectively. 
In particular, Theorem \ref{rectrancriter}(i) and (ii) imply that $(T^{\mu}_t)_{t>0}$ is recurrent if $d=2$ and transient if $d \geq 3$.

\item[(ii)]
Let $p\in (d, \infty)$, $\gamma:=2(1-\frac{d}{p})$ and define $\psi_0(x):=\|x\|^{\gamma}$, $x \in B_{1/4}$. Then $\psi_0 \in H^{1,p}(B_{1/4})$ with
$\nabla \psi_0(x) =\frac{\gamma}{\|x\|^{1-\gamma}} \frac{x}{\|x\|}$, $x \in B_{1/4}$. By \cite[Theorem 4.7]{EG15}, we can extend $\psi_0 \in H^{1,p}(B_{1/4})$ to $\psi \in H^{1,p}(\R^d) \cap C_0(\R^d)$ with $\psi \geq 0$ \,and\, $\text{\rm supp}(\psi) \subset B_{1/2}$. Now define
\begin{equation} \label{defnphi}
\phi(x):= 1+\sum_{k=0}^{\infty} \psi (x-k\mathbf{e}_1), \quad x \in \R^d.
\end{equation}
Then $\phi \in H^{1,p}_{loc}(\R^d) \cap C(\R^d)$ with $
1 \leq \phi(x) \leq \sup_{B_{1/2}}\psi$ for all $x \in \R^d$. Since $\nabla \phi$ has infinitely many singularities that form unbounded set in $\R^d$, so does $\langle \nabla \phi(x), x \rangle$. Now consider the situation of (i) with $\phi$ there replaced by the $\phi$ in \eqref{defnphi}. Then, \eqref{uplowcondit} holds, hence $\mu$ is the unique invariant measure for $(T^{\mu}_t)_{t>0}$ and the unique invariant measure for the family of strong solutions $(Y^y_t)_{t\geq0}$, $y \in \R^d$ to \eqref{mainsde2} in the sense of Theorem \ref{mainsto}(iii), even though \eqref{conscrit1} is not fulfilled.

\item[(iii)]
Let
$$
Lf = \frac12 \Delta f +\langle \frac12 \nabla C^T, \nabla f\rangle, \quad f \in C_0^{\infty}(\R^d),
$$
where $C=(c_{ij})_{1 \leq i,j \leq d}$ is an anti-symmetric matrix of functions that satisfies $c_{ij} \in H^{1,p}_{loc}(\R^d) \cap C(\R^d)$ with $p>d$ for all $1 \leq i,j \leq d$. Then, {\bf (S1)} holds with $A=id$, $\rho=1$, $\overline{\mathbf{B}}=0$ and $\beta^{\rho, A+C^T}=\frac12 \nabla C^T$, so that $\mu=dx$ is an infinitesimally invariant measure for $(L, C_0^{\infty}(\R^d))$. Since $\mu=dx$ is infinite, there exists no finite invariant measure for $(T^{\mu}_t)_{t>0}$ by \cite[Proposition 3.20(i)]{LT21in}. If there exist constants $M>0$, $N_0 \in \N$, such that
\begin{equation} \label{conscrit1*}
\langle \nabla C^T(x), x \rangle \leq M\|x\|^2 \left(\ln\|x\| +1 \right), \quad \text{$\mu$-a.e.}\; x \in \R^d \setminus B_{N_0}
\end{equation}
or there exists $\alpha \in [0,1)$ such that
\begin{equation}\label{conscrit2*}
\max_{1 \leq i,j \leq d}|c_{ij}(x)| \leq M \|x\|^{\alpha}  \quad \text{$\mu$-a.e. $x \in \R^d \setminus B_{N_0}$},
\end{equation}
then $(T^{\mu}_t)_{t>0}$ is conservative by \cite[Corollary 3.27]{LT20} or Theorem \ref{theo3:1}(i), respectively, hence it follows from \cite[Corollary 3.23, Theorem 3.52]{LT20} that for each $y \in \mathbb{R}^d$ and probability space $(\widetilde{\Omega}, \widetilde{\mathcal{F}},  \widetilde{\P})$ carrying a $d$-dimensional standard Brownian motion $(\widetilde{W}_t)_{t \geq 0}$
there exists a pathwise unique and strong solution  $(Y^y_t)_{t\geq 0}$ to
\begin{equation} \label{mainsde3} 
Y^y_t = y+ \widetilde{W}_t -\frac12 \int_0^t (\sum_{j=1}^d \partial_jc_{1j}, \ldots, \sum_{j=1}^d \partial_j c_{dj})(Y^y_s) ds, \quad 0 \leq t<\infty.
\end{equation}
Therefore, if either \eqref{conscrit1*} or \eqref{conscrit2*} holds, then there is no finite infinitesimally invariant measure for $(L, C_0^{\infty}(\R^d))$ by  \cite[Proposition 3.20(ii)]{LT21in}. 
In particular, if $\alpha=0$, then by Theorem \ref{uniqueinfinitesim}, Proposition \ref{cor:3.8} and Theorem \ref{mainsto}(iii),  $\mu$ is the unique infinitesimally invariant measure for $(L, C_0^{\infty}(\R^d))$ and the unique invariant measure for $(T^{\mu}_t)_{t>0}$ and the unique invariant measure for the family of strong solutions $(Y^y_t)_{t\geq0}$, $y \in \R^d$ to \eqref{mainsde2}  in the sense of Theorem \ref{mainsto}(iii).
\\
Now define $C=(c_{ij})_{1 \leq i,j \leq d}$ by 
$$
c_{1d}:=\phi, \; c_{d1}:=-\phi, \; \text{ and } c_{ij}:=0 \text{ for all } (i, j) \in \{1,\ldots, d\}^2 \setminus \{(1,d), (d,1) \},
$$
where $\phi$ is defined as in (ii). Then by Theorem \ref{uniqueinfinitesim} and Proposition  \ref{cor:3.8}(i), $\mu=dx$ is the unique infinitesimally invariant measure for $(L, C_0^{\infty}(\R^d))$ and $(L'^{, \mu}, C_0^{\infty}(\R^d))$ and the unique invariant measure for $(T^{\mu}_t)_{t>0}$ and $(T'^{, \mu}_t)_{t>0}$. In particular, by Theorem \ref{rectrancriter}(i) and (ii), $(T^{\mu}_t)_{t>0}$ is recurrent if $d=2$ and transient if $d \geq 3$. But note that $\langle \frac{1}{2} \nabla C^T(x), x\rangle$ has infinitely many singularities that form an unbounded set in $\R^d$, hence \eqref{conscrit1*} is not fulfilled, but \eqref{conscrit2*} holds.
\end{itemize}
\end{exam}
\text{}\\
{\bf Acknowledgement.} \; The author is deeply grateful to Professor Gerald Trutnau for valuable discussions and suggestions. Additionally, the author would like to express his great appreciation to the anonymous referees for their comments and suggestions.

\text{}\\

\

\text{}\\
\text{}\\
\centerline{}
Haesung Lee\\
Department of Mathematics and Computer Science, \\
Korea Science Academy of KAIST, \\
105-47 Baegyanggwanmun-ro, Busanjin-gu,\\
Busan 47162, Republic of Korea\\
E-mail: fthslt14@gmail.com
\end{document}